\documentclass[12pt]{amsart}
\usepackage{latexsym}
\usepackage{amssymb}
\usepackage{amsfonts}
\usepackage[pdftex]{hyperref}
\newtheorem{theorem}{Theorem}[section]
\newtheorem{lemma}[theorem]{Lemma}
\newtheorem{proposition}[theorem]{Proposition}

\theoremstyle{definition}

\newtheorem*{notation}{Notation}

\theoremstyle{remark}

\newcommand{\GL}{{\mathrm {GL}}}
\newcommand{\PGL}{{\mathrm {PGL}}}
\newcommand{\SL}{{\mathrm {SL}}}
\newcommand{\PSL}{{\mathrm {PSL}}}
\newcommand{\PSU}{{\mathrm {PSU}}}

\newcommand{\Ker}{\operatorname{Ker}}
\newcommand{\Aut}{{\mathrm {Aut}}}
\newcommand{\Out}{{\mathrm {Out}}}
\newcommand{\Mult}{{\mathrm {Mult}}}

\newcommand{\Irr}{{\mathrm {Irr}}}

\newcommand{\St}{{\mathrm {St}}}

\newcommand{\cd}{{\mathrm {cd}}}
\newcommand{\lcm}{{\mathrm {lcm}}}

\newcommand{\ZZ}{{\mathbb Z}}

\newcommand{\FF}{{\mathbb F}}

\newcommand{\ta}{\hspace{0.5mm}^{2}\hspace*{-0.2mm}}
\newcommand{\tb}{\hspace{0.5mm}^{3}\hspace*{-0.2mm}}

\begin{document}
\title[$\PSL_4(q)$ and character degrees]
{Projective Special  linear groups $\PSL_4(q)$ are determined by the set of their
character degrees }

\author[H.N. Nguyen]{Hung Ngoc Nguyen}
\address{Department of Mathematics, The University of Akron, Akron,
Ohio 44325, USA} \email{hungnguyen@uakron.edu}

\author[H.P. Tong-Viet]{Hung P. Tong-Viet}
\address{School of Mathematical Sciences, North-West University (Mafikeng),
Mmabatho 2735, South Africa} \email{tvphihung@gmail.com}

\author[T.P. Wakefield]{Thomas P. Wakefield}
\address{Department of Mathematics and Statistics, Youngstown State University, One University
Plaza, Youngstown, Ohio 44555, USA }\email{tpwakefield@ysu.edu}



\date{\today}

\begin{abstract} Let $G$ be a finite group and let $\cd(G)$ be the set
of all irreducible complex character degrees of $G$. It was conjectured
by Huppert in Illinois J. Math. 44 (2000) that, for every
non-abelian finite simple group $H$, if $\cd(G)=\cd(H)$ then $G\cong
H\times A$ for some abelian group $A$. In this paper, we confirm the
conjecture for the family of projective special linear groups
$\textrm{PSL}_4(q)$ with $q\geq 13$.
\end{abstract}

\maketitle


\section{Statement of the result}

Let $G$ be a finite group and let $\cd(G)$ denote the set of its
irreducible complex character degrees. In general, the character
degree set $\cd(G)$ does not completely determine the structure of
$G$. It is possible for non-isomorphic groups to have the same set
of character degrees. For example, the non-isomorphic groups $D_8$
and $Q_8$ not only have the same set of character degrees, but also
share the same character table. The character degree set also cannot
be used to distinguish between solvable and nilpotent groups, as
$\text{cd}(Q_8)=\cd(S_3)=\{1,2\}$. However, Huppert conjectured in
the late 1990s that the nonabelian simple groups are essentially
determined by the set of their character degrees. More explicitly,
he proposed in~\cite{Huppert1} that
\begin{quote} \emph{if $G$ is a finite group and $H$ a finite
nonabelian simple group such that the sets of character degrees of
$G$ and $H$ are the same, then $G \cong H \times A$, where $A$ is an
abelian group.}\end{quote} As abelian groups have only the trivial
character degree and the character degrees of $H \times A$ are the
products of the character degrees of $H$ and those of $A$, this
result is the best possible. The hypothesis that $H$ is a nonabelian
simple group is critical. There cannot be a corresponding result for
solvable groups since $Q_8 \ncong S_3 \times A$ for any abelian
group $A$.

To give some evidence, Huppert verified the conjecture on a
case-by-case basis for many nonabelian simple groups, including the
Suzuki groups, many of the sporadic simple groups, and a few of the
simple groups of Lie type, cf.~\cite{Huppert1}. There has been much
recent work in verifying the conjecture for simple groups of Lie
type, cf.~\cite{Huppert1,Huppert2,TV1,HW1,Wakefield1,Wakefield3}. In
particular, the conjecture has been verified for all simple groups
of Lie type of rank at most two.

In this paper, we establish Hupper's conjecture for the simple
linear groups $\PSL_4(q)$, which is the first time the conjecture is
verified for a family of simple Lie type groups of rank larger than
two. The case $\PSL_4(2)$, which is isomorphic to $A_8$, is already
done by Huppert in~\cite{Huppert9}. It turns out that the proof in
the case $3\leq q\leq 11$ requires some ad hoc arguments and, as the
paper is long enough, we decide to postpone these exceptional cases
to another time. Hence we assume $q\geq 13$ from now on.

\begin{theorem} Let $q\geq 13$ be a prime power and let $G$ be a finite group such that
$\cd(G)=\cd(\PSL_4(q))$. Then $G$ is isomorphic to the direct
product of $\PSL_4(q)$ and an abelian group.
\end{theorem}

Our proof is based on but requires modifications in the method
outlined by Huppert in~\cite{Huppert1}. We also use the structure of
maximal subgroups on finite classical groups in low dimension by
Kleidman and Liebeck in~\cite{Kleidman,Kleidman-Liebeck} and the
data given in Lubeck's webpage~\cite{Lubeck} on the complex
character degrees of finite groups of Lie type.

After a collection of some useful lemmas in
section~\ref{preparation1} and some analysis of character degrees of
$\PSL_4(q)$ in section~\ref{preparation2}, we prove in
section~\ref{step 1} that $G$ is quasi-perfect, i.e. $G'=G''$. We
then show in section~\ref{step 2} that if $G'/M$ is a chief factor
of $G$ then $G'/M\cong\PSL_4(q)$. It is then shown in
section~\ref{step3} that if $C/M=C_{G/M}(G'/M)$ then $G/M\cong
G'/M\times C/M$, which basically means that no outer automorphism of
$\PSL_4(q)$ is involved in the structure of $G$. In
section~\ref{step4}, we prove a technical result that every linear
character of $M$ is $G'$-invariant, which in particular implies that
$M'=[M,G']$ and $|M:M'|$ divides $|\Mult(G'/M)|$ by
Lemma~\ref{lemmaHuppert1}. Using this, we deduce in
section~\ref{step 5} that $M=1$ and therefore $G'\cong \PSL_4(q)$
and $G\cong G'\times C_G(G')$, as desired. On the way to the proof
of the main theorem, we also derive some interesting results on the
behavior of the irreducible characters of $\PSL_4(q)$ under the
action of its outer automorphism groups (see section~\ref{step3}).

\begin{notation} Our notation is fairly standard (see, e.g. \cite{Atl1} and \cite{Isaacs}). In particular, if $G$
is a finite group then $\Irr(G)$ denotes the set of all irreducible
characters of $G$ and $\cd(G)$ denotes the set of degrees of
irreducible characters in $\Irr(G)$. The set of all prime divisors of $|G|$ is denoted by $\pi(G).$ If $N\unlhd G$ and
$\lambda\in\Irr(N)$, then the induction of $\lambda$ from $N$ to $G$
is denoted by $\lambda^G$ and the set of irreducible constituents of
$\lambda^G$ is denoted by $\Irr(G|\lambda).$ Furthermore, if
$\chi\in\Irr(G)$ then $\chi_N$ is the restriction of $\chi$ to $N$.
The Schur multiplier of a group $G$ is denoted by $\Mult(G)$.
Finally, the greatest common divisor of two integers $a$ and $b$ is
$(a,b)$. \end{notation}


\section{Preliminaries}\label{preparation1}

In this section, we collect and also prove some lemmas which will be
used throughout the paper.

\begin{lemma}\emph{(\cite[Corollary~ 11.29]{Isaacs}).}\label{lemmaIsaacs} Let $N\unlhd G$ and $\chi\in\Irr(G)$. Let $\psi$ be a
constituent of $\chi_N$. Then $\chi(1)/\psi(1)$ divides $|G:N|$.
\end{lemma}

\begin{lemma}[Gallagher]\label{Gallagher} Let $N\unlhd G$ and $\chi\in \Irr(G)$. If
$\chi_N\in\Irr(N)$ then $\chi\tau\in\Irr(G)$ for every
$\tau\in\Irr(G/N)$.
\end{lemma}

\begin{lemma}[Thompson]\label{Thompson} Suppose that $p$ is a prime and $p\mid \chi(1)$ for
every nonlinear $\chi\in\Irr(G)$. Then $G$ has a normal
$p$-complement.
\end{lemma}

\begin{lemma}[\cite{Isaacs}, Lemma 12.3 and Theorem 12.4]\label{lemmaIsaacs1} Let $N\lhd G$ be maximal such that $G/N$ is solvable
and nonabelian. Then one of the following holds.
\begin{enumerate}
\item[(i)] $G/N$ is a $r$-group for some prime $r$. If $\chi\in \Irr(G)$ and
$r\nmid \chi(1)$, then $\chi \tau\in\Irr(G)$ for all
$\tau\in\Irr(G/N)$.
\item[(ii)] $G/N$ is a Frobenius group with an elementary abelian
Frobenius kernel $F/N$. Thus $|G:F|\in \cd(G)$, $|F:N|=r^a$ where
$a$ is the smallest integer such that $|G:F|\mid r^a-1$. For every
$\psi\in\Irr(F)$, either $|G:F|\psi(1)\in\cd(G)$ or $|F:N|\mid
\psi(1)^2$. If no proper multiple of $|G:F|$ is in $\cd(G)$, then
$\chi(1)\mid |G:F|$ for all $\chi\in\Irr(G)$ such that $r\nmid
\chi(1)$.
\end{enumerate}
\end{lemma}

\begin{lemma}\label{lemma1} In the context of $(ii)$ of
Lemma~\ref{lemmaIsaacs1}, we have
\begin{enumerate}
\item[(i)]  If $\chi\in\Irr(G)$ such that
$\lcm(\chi(1),|G:F|)$ does not divide any character degree of $G$,
then $r^a\mid \chi(1)^2$.
\item[(ii)] If $\chi\in\Irr(G)$ such that no proper multiple of
$\chi(1)$ is a degree of $G$, then either $|G:F|\mid \chi(1)$ or
$r^a\mid \chi(1)^2$
\end{enumerate}
\end{lemma}

\begin{proof} (i) Suppose that $\psi$ is a constituent of $\chi_F$.
Lemma~\ref{lemmaIsaacs} then implies that $\chi(1)/\psi(1)\mid
|G:F|$. In other words, $\chi(1)\mid |G:F|\psi(1)$. Since
$\lcm(\chi(1),|G:F|)$ does not divide any character degree of $G$,
we see that $|G:F|\psi(1)$ is not a degree of $G$. This forces $r^a
\mid \psi(1)^2$ by Lemma~\ref{lemmaIsaacs1}(ii). Since $\psi(1)$
divides $\chi(1)$, it follows that $r^a$ divides $\chi(1)^2$.

(ii) Let $\psi\in\Irr(F)$ be an irreducible constituent of $\chi_F$.
By Lemma~\ref{lemmaIsaacs1}(ii), either $|G:F|\psi(1)\in\cd(G)$ or
$r^a\mid \psi(1)^2$. If the latter case holds then we are done. So
we assume that $|G:F|\psi(1)\in\cd(G)$. Write $\chi(1)=k\phi(1)$ for
some integer $k$. It follows by Lemma~\ref{lemmaIsaacs} that $k\mid
|G:F|$. Therefore, $|G:F|\chi(1)/k=|G:F|\phi(1)\in\cd(G)$. As no
proper multiple of $\chi(1)$ is in $\cd(G)$, it follows that
$|G:F|=k$ and hence $|G:F|\mid \chi(1)$.
\end{proof}

\begin{lemma}\emph{(\cite[Lemma~5]{Bianchi}).}\label{lemma2} Let $N=S\times\cdot\cdot\cdot\times S$, a direct product
of $k$ copies of a nonabelian simple group $S$, be a minimal normal
subgroup of $K$. If $\chi\in\Irr(S)$ extends to $\Aut(S)$, then
$\chi(1)^k$ is a character degree of $K$.
\end{lemma}

\begin{proof} Note that $N$ can be considered as a subgroup of
$K/C_K(N)$ and $K/C_K(N)$ is embedded in $\Aut(N)=\Aut(S)\wr S_k$.
Let $\lambda$ be an extension of $\chi$ to $\Aut(S)$. Since
$\lambda\times\cdot\cdot\cdot\times \lambda$ is invariant under
$\Aut(S)\wr S_k$, it is extendable to $\Aut(S)\wr S_k$. Hence the
character $\chi\times\cdot\cdot\cdot\times \chi\in\Irr(N)$ is
extendable to $\Aut(S)\wr S_k$. In particular, it can be extended to
$K/C_K(N)$. The lemma follows.
\end{proof}

\begin{lemma}\emph{(\cite[Lemma~3]{Huppert1}).}\label{lemmaHuppert2} Let $M\unlhd
G$, $\theta\in\Irr(M)$, and $I=I_G(\theta)$. If $\phi\in\Irr(I)$
lying above $\theta$, then $\phi=\theta_0\tau$, where $\theta_0$ is
a character of an irreducible projective representation of $I$ of
degree $\theta(1)$ and $\tau$ is a character of an irreducible
projective representation of $I/M$.
\end{lemma}

\begin{lemma}\emph{(\cite[Lemma~6]{Huppert1}).}\label{lemmaHuppert1} Suppose $M\unlhd G'=G''$ and
$\lambda^g=\lambda$ for all $g\in G'$ and $\lambda\in\Irr(M)$ such
that $\lambda(1)=1$. Then $M'=[M,G']$ and $|M:M'|$ divides the order
of the Schur multiplier of $G'/M$.
\end{lemma}

The following result is an easy consequence of \cite[Theorem
$6.18$]{Isaacs}.

\begin{lemma}\label{fully ramified}
Let $M\unlhd L\unlhd G$ be normal subgroups of a group $G$ such that
$L/M$ is an abelian chief factor of $G.$ Let $\theta\in\Irr(M)$ such
that $\theta$ is $L$-invariant and let $\lambda$ be an irreducible
constituent of $\theta^L.$  Suppose that $\lambda(1)>\theta(1)$  and
$\lambda$ is $G$-invariant. Then
$\lambda(1)/\theta(1)=\sqrt{|L:M|}.$
\end{lemma}

\begin{proof} As $\theta$ is $L$-invariant, by Clifford theory, we have that $\lambda_M=e\theta$
for some integer $e.$ Since $\lambda(1)>\theta(1),$  we deduce that
$e>1.$ By \cite[Theorem $6.18$]{Isaacs} we obtain that $e^2=|L:M|$
and so $\lambda(1)=e\theta(1)=\sqrt{|L:M|}\theta(1)$ as required.
\end{proof}

\begin{lemma}\label{Schur cover} Let $S$ be a nonabelian simple group. Let $G$ be a perfect group so that $G/M\cong S$ and $|M|=|\Mult(S)|$ where $M$ is
cyclic. Then $G$ is isomorphic the Schur cover of $S$.
\end{lemma}

\begin{proof} As $M$ is normal in $G=G'$ and $M$ is cyclic, $C_G(M)$ is normal in $G$ and contains $M$. Since $G/M$ is simple, we obtain that $C_G(M)=G$ or $C_G(M)=M$.
If the former case holds, then $M\leq Z(G)\cap G'$ so that by
definition $G$ is the Schur cover of $G/M$. For the latter
possibility, $G/C_G(M)=G/M$ embeds into $\Aut(M)$, which is solvable
while $G/M$ is simple, a contradiction.
\end{proof}


\section{Character degrees of the simple linear groups
$\textrm{PSL}_4(q)$}\label{preparation2}

Let $\Phi_k$ denote the $k$th cyclotomic polynomial evaluated at
$q$. In particular, $$\Phi_1=q-1,\Phi_2=q+1,\Phi_3=q^2+q+1, \text{
and } \Phi_4=q^2+1.$$ The data in \cite{Lubeck} gives the character
degrees of $\textrm{SL}_4(q)$ and $\textrm{PGL}_4(q)$. From there,
we are able to extract the character degrees of $\PSL_4(q)$. These
degrees are given in Table~\ref{character}. The word ``possible'' in
the second column means that the condition for the existence of
corresponding degree is fairly complicated.

\begin{table}[h]\caption{Character degrees of $\PSL_4(q)$ (see \cite{Lubeck}).}\label{character}
\begin{tabular}{lll}
Degrees & Conditions\\\hline
$1$ & any $q$\\
 $q\Phi_3$ & any $q$\\
$\Phi_2\Phi_4$&$q=4$ or $q\geq 6$\\
 $\Phi_1^2\Phi_3$& any $q$\\
 $\frac{1}{2}\Phi_1^2\Phi_3$& possible \\
  $q^2\Phi_4$&any $q$\\
   $\Phi_3\Phi_4$&$q\geq4$\\
   $\frac{1}{2}\Phi_3\Phi_4$& possible \\
    $\Phi_1\Phi_3\Phi_4$&any $q$\\
     $q^3\Phi_3$&any $q$\\
      $q\Phi_3\Phi_4$&$q\geq3$\\
       $q\Phi_2^2\Phi_4$&$q=4$ or $q\geq 6$\\
       $\Phi_2\Phi_3\Phi_4$ & $q\geq4$\\
       $\Phi_1^3\Phi_2\Phi_3$ & any $q$\\
       $c\Phi_1^3\Phi_2\Phi_3$ & $c=1/2, 1/4$, possible\\\hline
\end{tabular}\quad\quad
\begin{tabular}{lll}
Degrees & Conditions\\\hline
       $q^2\Phi_1^2\Phi_3$ & any $q$\\
       $\frac{1}{2}q^2\Phi_1^2\Phi_3$ & possible\\
       $\Phi_1^2\Phi_3\Phi_4$ & $q\geq4$\\
       $c\Phi_1^2\Phi_3\Phi_4$ & $c=1/2, 1/4$, possible\\
       $\Phi_1^2\Phi_2^2\Phi_4$ & any $q$\\
       $q^6$ & any $q$\\
       $q\Phi_1\Phi_3\Phi_4$ & any $q$\\
       $\Phi_1\Phi_2\Phi_3\Phi_4$ & $q\geq3$\\
       $\frac{1}{2}\Phi_1\Phi_2\Phi_3\Phi_4$ & possible\\
       $q^3\Phi_2\Phi_4$ & $q=4$ or $q\geq 6$\\
       $q^2\Phi_3\Phi_4$ & $q\geq3$\\
       $\frac{1}{2}q^2\Phi_3\Phi_4$ & possible\\
       $q\Phi_2\Phi_3\Phi_4$ & $q\geq4$\\
       $\Phi_2^2\Phi_3\Phi_4$ & $q\geq7$\\
       $c\Phi_2^2\Phi_3\Phi_4$ & $c=1/2, 1/4$, possible\\\hline
\end{tabular}
\end{table}

We establish some arithmetic properties of character degrees of
$\textrm{PSL}_4(q)$, which will be needed in sections~\ref{step 1}
and~\ref{step 2}. Recall that a power is nontrivial power if it has
exponent greater than one. Also, a \emph{primitive prime divisor} of
$\Phi_k$ is a prime divisor of $\Phi_k$ that does not divide
$\Phi_i$ for every $1\leq i<k$. This prime exists by the classical
result of Zsigmondy in~\cite{Zsigmondy}.

\begin{lemma}\label{ppd} Let $\ell_i,i=1,2$ be primitive prime divisors of $\Phi_3$ and $\Phi_4,$ respectively.
If $\chi\in\Irr(\PSL_4(q))$ and $(\chi(1),\ell_1\ell_2)=1,$ then
$\chi(1)=q^6.$
\end{lemma}

\begin{proof} This is a straightforward check from
Table~\ref{character}.
\end{proof}

\begin{lemma}\label{maximal degree among degrees} The degrees
$\Phi_1^3\Phi_2\Phi_3$ and $\Phi_1^2\Phi_2^2\Phi_4$ are maximal with
respect to divisibility among the degrees of $\PSL_4(q)$.
\end{lemma}

\begin{proof} The degree $\Phi_1^3\Phi_2\Phi_3$ is maximal with
respect to divisibility among the degrees of $\PSL_4(q)$ since $
(q,\Phi_1^3\Phi_2\Phi_3)=1$ and
$(\Phi_4,\Phi_1^3\Phi_2\Phi_3)\mid2$. Also, the degree
$\Phi_1^2\Phi_2^2\Phi_4$ is maximal with respect to divisibility
among degrees of $\textrm{PSL}_4(q)$ since
$(q,\Phi_1^2\Phi_2^2\Phi_4)=(\Phi_3,\Phi_1^2\Phi_2^2\Phi_4)=1$.
\end{proof}

\begin{lemma}\label{helpful lemma} The following assertions hold.
\begin{enumerate}
\item[(i)]  If $\chi(1)$ and $\psi(1)$ are nontrivial degrees of
${\rm{PSL}}_4(q)$ such that $(\chi(1),\psi(1))$ $=1$, then the set
$\{\chi(1),\psi(1)\}$ contains at least one of the degrees:
$q\Phi_3$, $\Phi_1^2\Phi_3$, $\frac{1}{2}\Phi_1^2\Phi_3$,
$q^2\Phi_4$, $q^3\Phi_3$, $q^2\Phi_1^2\Phi_3$,
$\frac{1}{2}q^2\Phi_1^2\Phi_3$, $q^6$.

\item[(ii)] The only pairs of consecutive integers that are degrees
of  ${\rm{PSL}}_4(q)$ are $(q\Phi_3,$ $ \Phi_2\Phi_4)$ (when $q=4$
or $q\geq6$) and $(20,21)$ (when $q=2$).
\end{enumerate}
\end{lemma}

\begin{proof} (i) Assume by contradiction that neither $\chi(1)$ nor $\psi(1)$ is one of the degrees
$q\Phi_3$, $\Phi_1^2\Phi_3$, $\frac{1}{2}\Phi_1^2\Phi_3$,
$q^2\Phi_4$, $q^3\Phi_3$, $q^2\Phi_1^2\Phi_3$,
$\frac{1}{2}q^2\Phi_1^2\Phi_3$, $q^6$. Assume furthermore that both
of them are not in $\{c\Phi_1^3\Phi_2\Phi_3\mid c=1,1/2,1/4 \}$.
From the list of degrees of $\textrm{PSL}_4(q)$ we see that both
$\{\chi(1)$ and $\psi(1)\}$ are divisible by $\Phi_4$, which
violates the hypothesis. So at least one of $\{\chi(1)$ and
$\psi(1)\}$ is in $\{c\Phi_1^3\Phi_2\Phi_3\mid c=1,1/2,1/4 \}$. A
routine check gives the result.

(ii) This is obvious from Table~\ref{character}.
\end{proof}

\begin{lemma}\label{helpful lemma2} The following assertions hold.
\begin{enumerate}
\item[(i)] If $\chi\in\Irr({\rm{PSL}}_4(q))$ and
$\Phi_3\mid\chi(1)$ then $\chi(1)$ is not a nontrivial power.

\item[(ii)] The only possible degrees of ${\rm{PSL}}_4(q)$ which are
nontrivial powers are $\Phi_2\Phi_4$, $q\Phi_2^2\Phi_4$,
$\Phi_1^2\Phi_2^2\Phi_4$, $q^6$, $q^3\Phi_2\Phi_4$. In particular,
${\rm{PSL}}_4(q)$ has at most five nontrivial power degrees.
\end{enumerate}
\end{lemma}

\begin{proof} (i) Remark that
$$(q,\Phi_3)=(\Phi_2,\Phi_3)=(\Phi_4,\Phi_3)=1$$ and
$$(\Phi_1,\Phi_3)= 1 \text{ or } 3.$$ Therefore, if $\Phi_3\mid
\chi(1)$ and $\chi(1)$ is a nontrivial power then either $\Phi_3$ or
$\Phi_3/3$ is a nontrivial power. However, this is impossible
by~\cite{Nagell}.

(ii) This is a corollary of (i) and the fact that
$q^2\Phi_4=q^2(q^2+1)$ is not a nontrivial power by~\cite{Erdos}.
\end{proof}

\begin{lemma}\label{cd(SL4) is not in cd(PSL)} If $q$ is odd then
$$\Phi_1\Phi_2\Phi_3\Phi_4/2\in\cd(\SL_4(q)) \text{ but }
\Phi_1\Phi_2\Phi_3\Phi_4/2\notin\cd(\PSL_4(q)).$$
\end{lemma}

\begin{proof} Consider a semisimple element $s\in\GL_4(q)$ with
eigenvalues
$$\alpha,-\alpha,\omega^{(q+1)/2},-\omega^{(q+1)/2},$$ where $\alpha\in\FF_q\backslash\{0\}$ and $w$ is a generator of
$\FF_{q^2}\backslash\{0\}$. Then $$C_{\GL_4(q)}\cong
\GL_1(q)\times\GL_1(q)\times \GL_1(q^2).$$ Moreover,
$\overline{s}=sZ(\GL_4(q))$ is a semisimple element of $\PGL_4(q)$.
Note that if $\overline{t}\in C_{\PGL_4(q)}(\overline{s})$ then
$ts=st$ or $ts=-st$. Therefore,
$|C_{\PGL_4(q)}(\overline{s})|=2|C_{\GL_4(q)}|/(q-1)$. The
semisimple character $\chi_{\overline{s}}\in\Irr(\SL_4(q))$
associated to the conjugacy class of $\overline{s}$ has degree
$$\chi_{\overline{s}}(1)=\frac{|\SL_4(q)|_{p'}}{|C_{\PGL_4(q)}(\overline{s})|_{p'}}=\frac{\Phi_1\Phi_2\Phi_3\Phi_4}{2}.$$
Indeed, every irreducible character of $\SL_4(q)$ of degree
$\Phi_1\Phi_2\Phi_3\Phi_4/2$ is constructed in this way. Remark that
the determinant of $s$ is $\alpha^2\omega^{q+1}$. As $\omega^{q+1}$
is a generator of $\FF_q\backslash \{0\}$, $\alpha^2\omega^{q+1}\neq
1.$ In particular, $\overline{s}\notin \PGL_4(q)'$ and hence
$Z(\SL_4(q))\notin\Ker(\chi_{\overline{s}})$. We conclude that
$\chi_{\overline{s}}\notin\Irr(\PSL_4(q))$, which implies that
$\Phi_1\Phi_2\Phi_3\Phi_4/2\notin\cd(\PSL_4(q))$.
\end{proof}


\section{$G$ is quasi-perfect: $G'=G''$}\label{step 1}

Recall that $G$ is a finite group with the same character degree set
as $H=\textrm{PSL}_4(q)$. First, we show that $G'=G''$. Assume by
contradiction that $G'\neq G''$ and let $N\lhd G$ be maximal such
that $G/N$ is solvable and nonabelian. By Lemma~\ref{lemmaIsaacs1},
$G/N$ is an $r$-group for some prime $r$ or $G/N$ is a Frobenius
group with an elementary abelian Frobenius kernel $F/N$.

\medskip

\emph{Case 1.} $G/N$ is an $r$-group for some prime $r$. Since $G/N$
is nonabelian, there is $\theta\in\Irr(G/N)$ such that
$\theta(1)=r^b>1$. From the classification of prime power degree
representations of quasi-simple groups in \cite{Malle-Zalesskii}, we
deduce that $\theta(1)=r^b$ must be equal to the degree of the
Steinberg character of $H$ of degree $q^6$ and thus $r^b=q^6,$ which
implies that $r=p.$ By Lemma \ref{Thompson}, $G$ possesses a
nontrivial irreducible character $\chi$ with $p\nmid \chi(1).$
Lemma~\ref{lemmaIsaacs} implies that $\chi_N\in\Irr(N)$. Using
Gallagher's lemma, we deduce that $\chi(1)\theta(1)=q^6\chi(1)$ is a
degree of $G$, which is impossible.

\medskip

\emph{Case 2.} $G/N$ is a Frobenius group with an elementary abelian
Frobenius kernel $F/N$. Thus $|G:F|\in \cd(G)$, $|F:N|=r^a$ where
$a$ is the smallest integer such that $|G:F|\mid r^a-1$. Let $\chi$
be a character of $G$ of degree $q^6$. As no proper multiple of
$q^6$ is in $\cd(G)$, Lemma~\ref{lemma1} implies that either
$|G:F|\mid q^6$ or $r=p$. We consider the following subcases.

(a) $|G:F|\mid q^6.$ Then $|G:F|=q^6$ by Table~\ref{character}. This
means no multiple of $|G:F|$ is in $\cd(G)$. Therefore, by
Lemma~\ref{lemmaIsaacs1}, for every $\psi\in\Irr(G)$ either
$\psi(1)\mid q^6$ or $r\mid \psi(1)$. Taking $\psi$ to be characters
of degrees $q\Phi_3$ and $q^2\Phi_4$, we obtain that $r$ divides
both $\Phi_3$ and $\Phi_4$. This leads to a contradiction since
$(\Phi_3,\Phi_4)=1$.

(b) $r=p$. By Lemma~\ref{maximal degree among degrees} and the fact
that $p\nmid \Phi_1^3\Phi_2\Phi_3$ as well as
$\Phi_1^2\Phi_2^2\Phi_4$, we have $|G:F|$ divides both
$\Phi_1^3\Phi_2\Phi_3$ and $\Phi_1^2\Phi_2^2\Phi_4$. It follows that
$|G:F|$ is prime to $\ell_1\ell_2$ so that by Lemma \ref{ppd}, we
have $|G:F|=q^6$ as $|G:F|>1,$ which is impossible as $|G:F|$ is
prime to $p.$


\section{Eliminating the finite simple groups other than $\PSL_4(q)$}\label{step 2}

We have seen from section~\ref{step 1} that $G'=G''$. Moreover, it
is easy to see that $G'$ is nontrivial. Therefore, if $G'/M$ is a
chief factor of $G$, then $G'/M\cong S\times\cdot\cdot\cdot\times
S$, a direct product of $k$ copies of a nonabelian simple group $S$.
As $G'/M$ is a minimal normal subgroup of $G/M$ and
$\cd(G/M)\subseteq \cd(G)=\cd(\PSL_4(q))$, it follows by
Proposition~\ref{proposition} below that $k=1$ and $S\cong
\PSL_4(q)$. Equivalently, $G'/M\cong \PSL_4(q)$, as we wanted.

\begin{proposition}\label{proposition} Let $K$ be a group. Suppose that
$\cd(K)\subseteq\cd({\rm{PSL}}_4(q))$ and
$N=S\times\cdot\cdot\cdot\times S$, a direct product of $k$ copies
of a nonabelian simple group $S$, is a minimal normal subgroup of
$K$. Then $k=1$ and $S\cong {\rm{PSL}}_4(q)$.
\end{proposition}

The rest of this section is devoted to the proof of this
proposition. Indeed, the proof is the combination of
Lemmas~\ref{eliminate alternating}, \ref{eliminate sporadic},
\ref{eliminate classical groups}, and~\ref{eliminate exceptional
groups}. First, we need the following.

\begin{lemma}\label{sporadic groups} If $S$ is a sporadic group or
the Tits group, then there are at least six nontrivial irreducible
characters of $S$ of distinct degrees which are extendable to
$\Aut(S)$.
\end{lemma}

\begin{proof} This can be checked by using Atlas~\cite{Atl1}.
\end{proof}

\begin{lemma}\label{eliminate alternating} With the hypothesis of Proposition~\ref{proposition},
$S$ is not an alternating group of degree larger than $6$.
\end{lemma}

\begin{proof} Assume that $S=A_n$ with $n\geq 7$. It is well-known that the irreducible characters of $S_n$
are in one-to-one correspondence with partitions of size $n$. In
particular, $S_n$ has three degrees $\chi_1(1)=n(n-3)/2,$
$\chi_2(1)=(n-1)(n-2)/2$ and $\chi_3(1)=n-1$ corresponding to
partitions $(2^2,1^{n-4}),$ $(3,1^{n-3})$ and $(n-1,1),$
respectively. As $n\geq7$, these partitions are not self-conjugate
and hence $\chi_i,i=1,2,3$  are still irreducible when restricting
to $A_n$. Using Lemma~\ref{lemma2}, we deduce that
$\chi_i(1)^k\in\cd(K)$ and hence
$\chi_i(1)^k\in\cd(\textrm{PSL}_4(q))$ for $i=1,2,3$. Since
$(\chi_1(1)^k,\chi_2(1)^k)=1$, Lemma~\ref{helpful lemma}(i) implies
that the set $\{\chi_1(1)^k,\chi_2(1)^k\}$ contains at least one of
the degrees
$$q\Phi_3, \Phi_1^2\Phi_3, \frac{1}{2}\Phi_1^2\Phi_3,
q^2\Phi_4, q^3\Phi_3, q^2\Phi_1^2\Phi_3,
\frac{1}{2}q^2\Phi_1^2\Phi_3, q^6.$$ If $k\geq2$, since $q^6$ is the
only nontrivial power in this list by Lemma~\ref{helpful lemma2}(i),
we would have
$q^6\in\{\chi_1(1)^k,\chi_2(1)^k\}=\{(n(n-3)/2)^k,((n-1)(n-2)/2)^k\}$.
This is impossible when $n\geq7$ and $n\neq 9$. Now assume that
$n=9.$ In this case, we  have that $\chi_3(1)^k=8^k$ is a nontrivial
prime power degree of $\textrm{PSL}_4(q).$ However this is
impossible by applying~\cite{Malle-Zalesskii}.

We have shown that the only choice for $k$ is $1$. That means the
consecutive integers $n(n-3)/2$ and $(n-1)(n-2)/2$ are both degrees
of $\PSL_4(q)$. Lemma~\ref{helpful lemma}(ii) now yields
$$n(n-3)/2=q\Phi_3 \text { and } q=4 \text{ or } q\geq6$$ or $$n(n-3)/2=20 \text{ and }
q=2.$$ The latter case can be eliminated easily. So we assume that
$n(n-3)/2=q\Phi_3$ and $q\geq4$. Since
$\chi_3(1)=n-1\in\cd(\textrm{PSL}_4(q))$ and $n-1<n(n-3)/2$ and
$n(n-3/2)=q\Phi_3$ is the smallest nontrivial degree of
$\textrm{PSL}_4(q)$, we obtain a contradiction.
\end{proof}

\begin{lemma}\label{eliminate sporadic} With the hypothesis of Proposition~\ref{proposition},
$S$ is not a sporadic simple group nor the Tits group.
\end{lemma}

\begin{proof} Assume that $S$ is a simple sporadic
group or the Tits group. By Lemma~\ref{sporadic groups}, $S$ has at
least six nontrivial irreducible characters of distinct degrees
which are extendable to $\Aut(S)$. The $k$th powers of these degrees
will be degrees of $K$ by Lemma~\ref{lemma2}. As
$\cd(K)\subseteq\cd(\PSL_4(q))$, it follows that $\cd(\PSL_4(q))$
contains at least six nontrivial powers if $k\geq2$, which is
impossible by Lemma~\ref{helpful lemma2}(ii).

We have shown that if $S$ is a simple sporadic group or Tits group
then $k=1$.

If $S$ is one of $Ly$, $Th$,
$Fi'_{24}$, $B$, or $M$, then by Atlas again, we see that $S$ has at
least $33$ nontrivial irreducible characters of distinct degrees
which extend to $\Aut(S)$. This is a contradiction since $\PSL_4(q)$
has at most $32$ nontrivial degrees.

Assume next that $S$ is $O'N$.  By \cite{Atl1}, $O'N$ has an
irreducible character degree $10,$$944$ which is extendable to
$\Aut(S)$. Therefore, as $q\Phi_3$ is the smallest nontrivial degree
of $\cd(\PSL_4(q))$, we have $q\Phi_3\leq 10,$$944$. It follows that
$13\leq q\leq 19$  and these cases can be ruled out directly by
using the fact that $\{19,31\}\subseteq \pi(S)\subseteq
\pi(\PSL_4(q)).$

For the remaining cases, by inspecting Atlas \cite{Atl1}, we see
that $S$ has a nontrivial irreducible character of degree smaller
than $2379$ which is extendable to $\Aut(S)$. This degree is a
degree of $\PSL_4(q)$ by Lemma~\ref{lemma2}, which is a
contradiction as the smallest nontrivial degree of  $\PSL_4(q)$ is
$q\Phi_3$, which is at least $13(13^2+13+1)=2379$.
\end{proof}

\begin{lemma}\label{eliminate classical groups} With the hypothesis of Proposition~\ref{proposition},
$S$ is not a simple Lie type group of classical type except
$\PSL_4(q)$. Furthermore, if $S\cong \PSL_4(q)$ then $k=1$.
\end{lemma}

\begin{proof} Suppose that $S$ is a simple group of Lie
type $G(q_1)$ where $q_1=r^b$, a prime power. Let $\St$ denote the
Steinberg character of $S$. It is well-known (see \cite{Feit} for
instance) that $\St$ is extendable to $\Aut(S)$ and $\St(1)=|S|_r$,
the $r$-part of $|S|$. By Lemma~\ref{lemma2}, we get
$|S|_r^k\in\cd(\PSL_4(q))$. From the classification of irreducible
representations of quasi-simple groups of prime power degrees in
\cite[Theorem 1.1]{Malle-Zalesskii}, we have that $q^6$ is the only
degree of $\PSL_4(q)$ that is a prime power and therefore
$|S|_r^k=q^6$ and $r=p$.

Consider a nontrivial character $\tau\in\Irr(S)$ different from
$\St$. Then $\tau\times \St\times\cdot\cdot\cdot\times
\St\in\Irr(N)$ and hence $\tau(1)\St(1)^{k-1}=\tau(1)|S|_r^{k-1}$
divides some degree of $K$. Inspecting the list of degrees of
$\PSL_4(q)$, we then deduce that $|S|_r^{k-1}\leq q^3$. This
together with the fact $|S|_r^k=q^6$ imply that $k\leq2$.

We wish to show that $k=1.$ By way of contradiction, assume that
$k=2.$ Hence $|S|_p^2=q^6.$ Let $C=C_K(N).$ Then $K/C$ embeds into
$\Aut(N)\cong \Aut(S)\wr \ZZ_2.$ Let $B=\Aut(S)^2\cap K.$ Then
$K/B\cong \ZZ_2.$ Let $\varphi=1\times St\in\Irr(N).$ We have that
$\varphi$ extends to $B$ and so $B$ is the inertia group of
$\varphi$ in $K.$ It follows that $\varphi^K(1)=2\varphi(1)=2|S|_p$
is a degree of $K$ and then $2|S|_p\in\cd(H).$ As this degree is
prime to $\ell_1\ell_2,$ Lemma \ref{ppd} implies that
$2|S|_p=|S|_p^2.$ It follows that $|S|_p=2,$ which is impossible.

We consider the following cases. In all cases except $\PSL_4(q)$,
we will reach a contradiction by finding a character
$\chi\in\Irr(S)$ extendable to $\Aut(S)$ so that
$\chi(1)\notin\cd(\PSL_4(q))$ or a character $\psi\in\Irr(S)$ so
that $\psi(1)$ divides no degree of $\PSL_4(q)$.

\begin{table}[h]\caption{Some unipotent characters of simple groups of Lie type \cite{Carter}.}\label{unipotent-character}
\begin{tabular}{ll}
$S=G(q)$ & $p$-part of degrees\\\hline
$\PSL^\pm_n(q)$ & $q^{(n-1)(n-2)/2}$\\
 $\mathrm{PSp}_{2n}(q)$ or $\mathrm{P\Omega}_{2n+1}(q)$, $p=2$ & $q^{(n-1)^2}/2$ \\
 $\mathrm{PSp}_{2n}(q)$ or $\mathrm{P\Omega}_{2n+1}(q)$, $p>2$ & $q^{(n-1)^2}$ \\
$\mathrm{P\Omega}^+_{2n}(q)$&$q^{n^2-3n+3}$\\
$\mathrm{P\Omega}^-_{2n}(q)$&$q^{n^2-3n+2}$\\
 $\tb D_4(q)$& $q^7$\\
 $\ta F_4(q)$& $q^6\sqrt{q/2}$ \\
  $\ta E_6(q)$& $q^{25}$\\
   $F_4(q)$, $p=2$&$q^{13}/2$\\
   $F_4(q)$, $p>2$&$q^{13}$\\
   $E_6(q)$& $q^{25}$ \\
    $E_7(q)$& $q^{46}$\\
     $E_8(q)$& $q^{91}$\\
             \hline
\end{tabular}
\end{table}

\medskip

(i) $S\cong \PSL_n^\pm(q_1)$ where $q_1=p^b$. Then we get
\begin{equation}\label{1}q^6=|S|_p=p^{bn(n-1)/2}.\end{equation} From Table~\ref{unipotent-character}, $S$ has a unipotent
character $\chi$ different from the Steinberg character with
$\chi(1)_p=p^{b(n-1)(n-2)/2}$. The fact that
$\chi(1)\in\cd(\PSL_4(q))$ and Table~\ref{character} then yield
\begin{equation}\label{2}p^{b(n-1)(n-2)/2}\mid q^3
\end{equation}
Now (\ref{1}) and (\ref{2}) imply $2(n-2)\leq n$ and hence $n\leq4$.

If $n=2$ then $q_1=q^6$. The fact that $\PSL_2(q_1)$ has an
irreducible of degree $q_1+1$ (see \cite{Lubeck}) then leads to a
contradiction since both $q^6+1$   divides no degree of $PSL_4(q)$.
If $n=3$ then $q_1^{3}=q^6$. We see that $q_1=q^2.$ This case does
not happen since $\PSL^\pm_3(q^2)$ has an irreducible character of
degree $(q^2\mp1)^2(q^2\pm1)$ which does not divide any degree of
$\PSL_4(q)$. Finally, if $n=4$ then $q_1^{6}=q^6$ and hence $q_1=q$.
The case $S\cong \rm{PSU}_4(q)$ does not happen since $\PSU_4(q)$
has a unipotent character of degree $q^3\Phi_6\notin\cd(\PSL_4(q))$
(see \cite[Table 20]{Clasen}). We are left with the case  $S\cong
\PSL_4(q)$, as wanted.

\medskip

(ii) $S\cong \mathrm{P\Omega}_{2n+1}(q_1)$ or $\mathrm{PSp}_{2n}(q_1)$
where $q_1=p^b$ and $n\geq2$. Then we get
\begin{equation}\label{3}q^6=|S|_p=p^{bn^2}.\end{equation} From Table~\ref{unipotent-character}, $S$ has a unipotent
character $\chi$ different from the Steinberg character with
$\chi(1)_p=p^{b(n-1)^2-1}$ when $p=2$ or $\chi(1)_p=p^{b(n-1)^2}$
when $p>2$. The fact that $\chi(1)\in\cd(\PSL_4(q))$ and
Table~\ref{character} then yield
\begin{equation}\label{4}p^{(b(n-1)^2-1)}\mid q^3
\end{equation}
Now (\ref{3}) and (\ref{4}) imply $2(b(n-1)^2-1)\leq bn^2$ and hence
$n\leq4$.

If $n=2$ then $q_1^{4}=q^6$ and hence $q_1^{2}=q^3$. Therefore
$q_1=q^{3/2}.$ This case cannot happen since $\mathrm{PSp}_4(q_1)$
has a unipotent character of degree $q_1(q_1-1)^2/2$ (see
\cite[Table 23]{Clasen}) and
$q_1^2(q_1-1)^2/2=q^{3}(q^{3/2}-1)^2/2\notin\cd(\PSL_4(q))$. If
$n=3$ then $q_1^{9}=q^6$. We see that $q_1=q^{2/3}.$ Taking a
unipotent character of $S$ of degree $q_1(q_1-1)^2(q_1^2+q_1+1)$
(see \cite[Table 24]{Clasen}) and arguing as in the case $n=2$, we
get a contradiction. Finally, if $n=4$ then $q_1^{16}=q^6$. We see
that $q_1=q^{3/8}.$ Taking a unipotent character of $S$ of degree
$q_1(q_1^2-q_1+1)(q_1^4+1)$ (see \cite[Table 25]{Clasen}) and
arguing as in the case $n=2$, we get a contradiction.

\medskip

(iii) $S\cong P\Omega_{2n}^\pm(q_1)$ where $q_1=p^b$ and $n\geq4$.
Then we get
\begin{equation}\label{5}q^6=|S|_p=p^{bn(n-1)}.\end{equation} From Table~\ref{unipotent-character}, $S$ has a unipotent
character $\chi$ different from the Steinberg character with
$\chi(1)_p=p^{b(n^2-3n+3)}$ when $S\cong
\mathrm{P}\Omega_{2n}^+(q_1)$ and $\chi(1)_p=p^{b(n^2-3n+2)}$ when
$S\cong \mathrm{P}\Omega_{2n}^-(q_1)$. The fact that
$\chi(1)\in\cd(\PSL_4(q))$ and Table~\ref{character} then yield
\begin{equation}\label{6}p^{b(n^2-3n+2)}\mid q^3
\end{equation}
Now (\ref{5}) and (\ref{6}) imply $2(n^2-3n+2)\leq n(n-1)$ and hence
$n=4$. Furthermore, $S\cong \mathrm{P}\Omega_{8}^-(q_1)$.

Since $n=4$, we have $q_1^{12}=q^6$. Therefore, $q_1=q^{1/2}.$
Taking a unipotent character of $S$ of degree $q_1(q_1^4+1)$ (see
\cite[Table 29]{Clasen}) and arguing as in the previous case, we get
a contradiction.
\end{proof}

\begin{lemma}\label{eliminate exceptional groups} With the hypothesis of Proposition~\ref{proposition},
$S$ is not a simple Lie type group of exceptional type.
\end{lemma}

\begin{proof} It follows from the proof of the previous lemma that $k=1.$
First we eliminate the case $S\cong \ta B_2(q_1)=\ta B_2(2^{2m+1})$
where $m\geq1$. Assume that $2^{2(2m+1)}=q^6$ and hence
$2^{(2m+1)}=q^3$. Also, $S$ has a unipotent character of degree
$\sqrt{q_1/2}(q_1-1)=2^m(2^{2m+1}-1)$ and therefore
$2^{m}(2^{2m+1}-1)\in\cd(\PSL_4(q))$. It follows that $2^{m}\mid
q^3$. Together with the equality $2^{2m+1}=q^3$, we obtain $m=1$.
This case can be ruled out easily by using Atlas. From now on we
assume that $S$ is one of the simple groups of exceptional Lie type
different from $\ta B_2$.

\medskip

(i) $S\cong G_2(q_1)$ where $q_1=p^b$. Then we have
$q_1^6=|S|_p=q^6$ and hence $q_1=q$. That means $S\cong G_2(q)$ and
$S$ has a unipotent character of degree $q\Phi_3\Phi_6/3$ (see
\cite[Table 33]{Clasen}). However, this is not a degree of
$\PSL_4(q)$.

\medskip

(ii) $S\cong \ta G_2(q_1)=\ta G_2(3^{2m+1})$ where $m\geq 1$. Then
we have $q_1^3=|S|_p=q^6$ and hence $3^{2m+1}=q_1=q^2$, which is
impossible.

\medskip

(iii) For the remaining groups, $S$ has a unipotent character (see
Table~\ref{unipotent-character}) different from the Steinberg
character whose degree has $p$-part larger than
$\sqrt{|S|_p}=\sqrt{q^6}=q^3$. By Lemma~\ref{lemma2}, this degree is
also a degree of $K$ and therefore a degree of $\PSL_4(q)$. However,
all degrees different from $q^6$ of $\PSL_4(q)$ has $p$-part at most
$q^3$ by Table~\ref{character}. This leads to a contradiction.
\end{proof}


\section{Outer automorphisms of $\PSL_4(q)$}\label{step3}
In section~\ref{step 2}, we have shown that $G'/M\cong\PSL_4(q)$. In
this section, we will prove that
$$G/M\cong G'/M\times C/M,$$ where $C/M=C_{G/M}(G'/M)$. Remark that $\cd(G/M)\subseteq \cd(\PSL_4(q))$.
For simple notation, we can assume that $M=1$ but only assume
$\cd(G)\subseteq \cd(\PSL_4(q))$. In other words, we have to prove
the following:

\begin{proposition}\label{step 4} Let $G$ be a finite group such
that $G'\cong\PSL_4(q)$ and $\cd(G)\subseteq\cd(\PSL_4(q))$. Then
$G\cong G'\times C_G(G')$.
\end{proposition}

\begin{proof} Assume the contrary that $G'\times C_G(G')$ is a proper subgroup of
$G$. Then $G$ induces on $G'$ some outer automorphism $\alpha$. Let
$q=p^f$. It is well known (cf. \cite[Theorem 2.5.12]{Gorenstein})
that $$\Out(G')=\langle d\rangle: (\langle \sigma\rangle\times
\langle \tau\rangle),$$ where $d$ is a diagonal automorphism of
degree $(q-1,4)$, $\sigma$ is the automorphism of $G'$ of order $f$
induced by the field automorphism $x\mapsto x^p$, and $\tau$ is the
inverse-transpose. Moreover, $\tau$ inverts the cyclic group
$\langle d\rangle$.

\medskip

(i) First we consider the case where $G/C_G(G')$ possesses only
inner and diagonal automorphisms. This means $q$ must be odd.

\textbf{Case} $\mathbf{q\equiv 3\bmod{4}}$: We then have
$$G/C_G(G')\cong \PGL_4(q).$$ From the list of degrees of $\PGL_4(q)$
and $\SL_4(q)$ in \cite{Lubeck}, we see that $\PGL_4(q)$ has $q-2$
irreducible characters of degree $\Phi_3\Phi_4$ while $\SL_4(q)$ has
$(q-3)/2$ characters of degree $\Phi_3\Phi_4$ and $2$ characters of
degree $\Phi_3\Phi_4/2$. It follows that $\PSL_4(q)$ has exactly two
irreducible characters of degree $\Phi_3\Phi_4/2$ and they are fused
under $\PGL_4(q)$.

We will get to a contradiction by showing that $\Phi_3\Phi_4/2$ is
not a degree of $G$. Assume by contrary that $\chi$ is an
irreducible character of $G$ of such degree. As $G/C_G(G')\cong
\PGL_4(q)$, which has no irreducible character of degree
$\Phi_3\Phi_4/2$, $\chi\notin\Irr(G/C_G(G'))$. Assume that $C_G(G')$
is not abelian and let $\lambda$ be a nonlinear irreducible
character of $C_G(G')$. Then $\St_{G'}\times \lambda\in\Irr(G'\times
C_G(G'))$ where $\St_{G'}$ is the Steinberg character of
$G'=\PSL_4(q)$. This is a contradiction since
$\St_{G'}(1)\lambda(1)=q^6\lambda(1)$ divides none of the degree of
$G$.

We have shown that $C_G(G')$ is abelian. Let $\psi\times \theta$ an
irreducible character $G'\times C_G(G')$ lying under $\chi$. As
$C_G(G')$ is abelian, we know $\theta(1)=1$. Assume that $\theta$ is
not $G$-invariant. As $G'$ centralizes $C_G(G')$ and $|G: G'\times
C_G(G')|=2$, it follows that $I_G(\theta)=G'\times C_G(G')$. Let
$\theta_0=1_{G'}\times \theta\in\Irr(I_G(\theta))$. Then $\theta_0$
is an extension of $\theta$ to $I_G(\theta)$, and therefore
$\theta_0^G\in\Irr(G)$ by Clifford theory. So
$$\theta_0^G(1)=|G:(G'\times C_G(G'))|\theta_0(1)=2$$ is a degree of
$G$, which is impossible. Thus $\theta$ is $G$-invariant.

If $\psi$ is $G$-invariant then $\psi$, considered as a character of
$G'C_G(G')/C_G(G')$, is also $G/C_G(G')$-invariant. As
$|G:G'C_G(G')|=2$, we deduce from \cite[Corollary 6.20]{Isaacs} that
$\psi$ extends to $G/C_G(G')$. In particular,
$$\psi(1)\in\cd(G/C_G(G'))=\cd(\PGL_4(q)).$$ On the other hand, we
also have that $\psi\times\theta$ is $G$-invariant and hence it
extends to $G$. So $$\chi(1)=(\psi\times\theta)(1)=\psi(1)$$ and
therefore $\Phi_3\Phi_4/2$ is a degree of $\PGL_4(q)$, a
contradiction. Thus $\psi$ is not $G$-invariant and hence
$\psi\times \theta$ is not $G$-invariant and $I_G(\psi\times
\theta)=G'\times C_G(G')$. Using Clifford theory again, we obtain
$(\psi\times\theta)^G\in\Irr(G)$ and hence
$$\chi(1)=(\psi\times\theta)^G(1)=2\psi(1).$$ Recall that
$\chi(1)=\Phi_3\Phi_4/2$. We then have $\psi(1)=\Phi_3\Phi_4/4$ is a
degree of $G'=\PSL_4(q)$, which is impossible.

\textbf{Case} $\mathbf{q\equiv 1\bmod{4}}$: In this case, $d$ is a
diagonal automorphism of $G'$ of degree $4$. We then have
$$G/C_G(G')\cong G':\langle d^2\rangle \text{ or } \PGL_4(q).$$ From the list
of degrees of $\PGL_4(q)$ and $\SL_4(q)$ in \cite{Lubeck}, we see
that $G'$ has exactly four irreducible characters of degree
$\Phi_1^2\Phi_3\Phi_4/4$, which are fused under $G':\langle
d^2\rangle$ to form two characters of degree
$\Phi_1^2\Phi_3\Phi_4/2$ and under $\PGL_4(q)$ to form one character
of degree $\Phi_1^2\Phi_3\Phi_4$.

We will get to a contradiction by showing that
$\Phi_1^2\Phi_3\Phi_4/4$ is not a degree of $G$. Assume by contrary
that $\chi$ is an irreducible character of $G$ of such degree. As
$G/C_G(G')\cong G':\langle d^2\rangle$ or $\PGL_4(q)$, which has no
irreducible character of degree $\Phi_1^2\Phi_3\Phi_4/4$ as $q\neq
5$ (when $q=5$, there is a coincidence
$\Phi_1^2\Phi_3\Phi_4/4=\Phi_1\Phi_3\Phi_4\in\cd(\PGL_4(q))$),
$\chi\notin\Irr(G/C_G(G'))$. Assume that $C_G(G')$ is not abelian
and let $\lambda$ be a nonlinear irreducible character of $C_G(G')$.
Then $$\St_{G'}\times \lambda\in\Irr(G'\times C_G(G')),$$ where
$\St_{G'}$ is the Steinberg character of $G'=\PSL_4(q)$. This is a
contradiction since $\St_{G'}(1)\lambda(1)=q^6\lambda(1)$ divides
none of the degree of $G$.

We have shown that $C_G(G')$ is abelian. Let $\psi\times \theta$ be
an irreducible character $G'\times C_G(G')$ lying under $\chi$. As
$C_G(G')$ is abelian, we know $\theta(1)=1$. Assume that $\theta$ is
not $G$-invariant. As $G'$ centralizes $C_G(G')$ and $|G: G'\times
C_G(G')|=4$ or $2$, it follows that
$$|I_G(\theta):(G'\times C_G(G'))|=1 \text{ or } 2 \text{ if } |G: G'\times
C_G(G')|=4$$ and $$|I_G(\theta):(G'\times C_G(G'))|=1 \text{ if }
|G: G'\times C_G(G')|=2.$$ Let $\theta_0$ be an extension of
$1_{G'}\times \theta\in\Irr(G'\times C_G(G'))$ to $I_G(\theta)$.
Then $\theta_0^G\in\Irr(G)$ by Clifford theory and hence
$$\theta_0^G(1)=|G:I_G(\theta)|\theta_0(1)=2 \text{ or } 4,$$ which is a degree of
$G$, a contradiction. Thus $\theta$ is $G$-invariant.

If $\psi$ is $G$-invariant then $\psi$, considered as a character of
$G'C_G(G')/C_G(G')$, is also $G/C_G(G')$-invariant. As
$|G:G'C_G(G')|$ is cyclic, we deduce from \cite[Corollary
11.22]{Isaacs} that $\psi$ extends to $G/C_G(G')$. In particular,
$\psi(1)\in\cd(G/C_G(G'))$. On the other hand, we also have that
$\psi\times\theta$ is $G$-invariant and hence it extends to $G$. So
$$\chi(1)=(\psi\times\theta)(1)=\psi(1).$$ This means
$\Phi_1^2\Phi_3\Phi_4/4$ is a degree of $G/C_G(G')$, which is
$G':\langle d^2\rangle$ or $\PGL_4(q)$, a contradiction. So $\psi$
is not $G$-invariant and hence $\psi\times \theta$ is not
$G$-invariant and therefore $$|I_G(\psi\times \theta):G'\times
C_G(G')|=1 \text{ or } 2.$$ Let $\varphi$ be an extension of
$\psi\times \theta$ to $I_G(\psi\times \theta)$. Using Clifford
theory, we obtain $(\varphi)^G\in\Irr(G)$ and hence
$$(\varphi)^G(1)=|G:I_G(G'\times C_G(G'))|\psi(1),$$ which is
$2\psi(1)$ or $4\psi(1)$. Since $\chi$ lies over $\psi\times\theta$,
it follows that $\chi(1)=2\psi(1)$ or $4\psi(1)$. Recall that
$\chi(1)=\Phi_1^2\Phi_3\Phi_4/4$. So
$\psi(1)=\Phi_1^2\Phi_3\Phi_4/8$ or $\Phi_1^2\Phi_3\Phi_4/16$. This
contradicts the fact that $\psi(1)$ is a degree of $G'=\PSL_4(q)$.

\medskip

(ii) Next we consider the case $\alpha=d^a\sigma^b\tau^c$ where
$0\leq a\leq (q-1,4)$, $0< b<f$, and $0\leq c\leq 1$. By a classical
result of Zsigmondy (cf.~\cite{Zsigmondy}), we can choose
$\omega\in\FF$ of order which is a primitive prime divisor of
$p^{4f}-1=q^4-1$, that is a prime divisor of $p^{4f}-1$ that does
not divide $\prod_{i=1}^{4f-1}(p^i-1)$. We then choose a semisimple
element $s\in \SL_4(q)$ with eigenvalues
$\omega,\omega^q,\omega^{q^2}$, and $\omega^{q^3}$. The image of $s$
under the canonical projection $\GL_4(q)\rightarrow \PGL_4(q)$ is a
semisimple element of $\PGL_4(q)$. Abusing the notation, we denote
it by $s$. Since $s,s^{-1}$, and $\tau(s)$ are all conjugate in
$\PGL_4(q)$, the semisimple character $\chi_s\in\Irr(\SL_4(q))$ of
degree $\Phi_1^3\Phi_2\Phi_3$ is real by~\cite[Lemma~2.5]{Dolfi} and
hence
$$(\chi_s)^\tau=\chi_{\tau(s)}=\chi_{s^{-1}}=\overline{\chi}_s=\chi_s$$
by~\cite[Corollary~2.5]{Navarro}. In other words, $\chi_s$ is
invariant under $\tau$.

Since $\PGL_4(q)$ has no degree which is a proper multiple of
$\Phi_1^3\Phi_2\Phi_3$, we obtain that $\chi_s$ is also
$d$-invariant. By checking the multiplicities of character degrees
of $\SL_4(q)$ and $\PGL_4(q)$, we see that there exists an $s$ as
above so that $\chi_s\in\Irr(\PSL_4(q))$.
Using~\cite[Lemma~2.5]{Dolfi} again, we have furthermore that
$\chi_s$ is not $\sigma^b$-invariant since $|s|=|\omega|$ does not
divide $|\PGL_4(p^b)|$. We have shown that $\chi_s$ is not
$\alpha$-invariant. Therefore $G$ has a degree which is a proper
multiple of $\chi_s(1)=\Phi_1^3\Phi_2\Phi_3$, a contradiction.

\medskip

(iii) Finally we consider the case $\alpha=d^a\tau$ where $0\leq
a\leq (q-1,4)$. Now we choose $\omega\in\FF$ of order which is a
primitive prime divisor of $p^{3f}=q^3-1$, that is a prime divisor
of $p^{3f}-1$ that does not divide $\prod_{i=1}^{3f-1}(p^i-1)$. We
then choose a semisimple element $s\in \SL_4(q)$ with eigenvalues
$1,\omega,\omega^q$, and $\omega^{q^2}$. The image of $s$ under the
canonical projection $\GL_4(q)\rightarrow \PGL_4(q)$ is a semisimple
element of $\PGL_4(q)$. Abusing the notation, we denote it by $s$.
It is routine to check that $\omega^{-1}$ is not an eigenvalue of
$s$. In particular, $s$ is not conjugate to $s^{-1}$ in $\PGL_4(q)$
and so $s$ is not real in $\PGL_4(q)$. It follows by~\cite[Lemma
2.5]{Dolfi} that the semisimple character $\chi_s\in \Irr(\SL_4(q))$
of degree $\Phi_1^2\Phi_2^2\Phi_4$ is not real. Moreover, $\chi_s$
is trivial at $Z(\SL_4(q))$ and therefore
$\chi_s\in\Irr(\PSL_4(q))$. Now using \cite[Corollary 2.5]{Navarro}
and the fact that $\tau(s)$ is conjugate to $s^{-1}$ in $\PGL_4(q)$,
we get
$$(\chi_s)^\tau=\chi_{\tau(s)}=\chi_{s^{-1}}=\overline{\chi}_s.$$
Therefore $\chi_s$ is moved by $\tau$ since it is not real. This is
a contradiction since $G$ does not have any degree which is a proper
multiple of $\chi_s(1)=\Phi_1^2\Phi_2^2\Phi_4$.
\end{proof}


\section{Linear characters of $M$ are $G'$-invariant} \label{step4}

We assume in this section that $q\geq 13.$
In this step, we prove that if $\theta\in \Irr(M)$ and $\theta(1)=1$
then $I:=I_{G'}(\theta)=G'$. Assume by contradiction that $I<G'$,
then $I\leq U<G'$ for some maximal subgroup $U$ of $G'$. Suppose
that
$$\theta^I=\sum _i\phi_i, \text{ where }\phi_i\in\Irr(I).$$ We then have
$\phi_i^{G'}\in\Irr(G')$ and hence $\phi_i(1)|G':I|\in\cd(G')$. It
follows that
$$\phi_i(1)|G':U||U:I| \text{ divides some degree of }G.$$
For simpler notation, let $t=|U:I|$, we have
$$t\phi_i(1)|G':U|\text{ divides some degree of }G.$$ In
particular, the index of $U/M$ in $G'/M$ divides some degree of $G$.

The maximal subgroups of $\textrm{PSL}_4(q)$ as well as other groups
in this paper are determined by Kleidman
in~\cite{Kleidman,Kleidman-Liebeck} and given in the following
tables. In these tables, $d=(q-1,4)$ and the symbol $^\wedge$ means
we are giving the structure of the preimage in special linear or
symplectic groups. Let $\tau$ be the natural projection from
$\SL_4(q)$ to $\PSL_4(q)$. The following lemma will eliminate most
of possibilities of $U/M$.

\begin{table}[h]\caption{Maximal subgroups of $\textrm{PSL}_4(q)$ (see \cite{Kleidman}).}\label{maximal-subgroups}
\begin{tabular}{lll}
Subgroup & Condition & Index\\\hline
$^\wedge[q^3]:\textrm{GL}_3(q)$ & & $\Phi_2\Phi_4$ $$\\
 $^\wedge[q^4]:(\textrm{SL}_2(q)\times
 SL_2(q)).\ZZ_{q-1}$&&$\Phi_3\Phi_4$\\
$^\wedge(\ZZ_{q-1})^3.S_4$&$q\geq 5$&$q^6\Phi_2^2\Phi_3\Phi_4/24$\\
 $^\wedge(\textrm{SL}_2(q)\times \textrm{SL}_2(q)).\ZZ_{q-1}.2$&$q\geq4$&$q^4\Phi_3\Phi_4/2$\\
  $(\textrm{PSL}_2(q^2)\times\ZZ_{q+1}).2$& &$q^4\Phi_1^2\Phi_3(q-1,2)/2d$\\
   $\textrm{PSL}_4(q_0).(q_0-1,4)$
   &$q=q_0^b$, $b$ prime & $\frac{|\PSL_4(q)|}{(q_0-1,4)|\textrm{PSL}_4(q_0)|}$\\
    $2^4. S_6$&$q=p\equiv1\bmod{8}$ &$|\textrm{PSL}_4(q)|/11520$\\
     $2^4.A_6$&$q=p\equiv5\bmod{8}$ &$|\textrm{PSL}_4(q)|/5760$\\

      $^\wedge\textrm{Sp}_4(q)\cdot(q-1,2)$&$$ &${q^2\Phi_1\Phi_3}/{(q-1,2)}$\\

        $\textrm{PSO}^+_4(q).2$&$q$ odd &$q^4\Phi_1\Phi_3\Phi_4/d$\\
         $\textrm{PSO}^-_4(q).2$&$q$ odd &$q^4\Phi_1^2\Phi_2\Phi_3/d$\\
         $\textrm{PSU}_4(q_0).(q-1,2)$&$q=q_0^2$ &$\frac{|\textrm{PSL}_4(q)|}{(q-1,2)|\textrm{PSU}_4(q_0)|}$\\
         $A_7$&$q=p\equiv1,2,4\bmod{7}$ &$|\textrm{PSL}_4(q)|/2520$\\
         $\textrm{PSU}_4(2)$&$q=p\equiv1\bmod{6}$ &$|\textrm{PSL}_4(q)|/25920$\\\hline
\end{tabular}
\end{table}

\begin{table}[h]\caption{Maximal subgroups of $\PSL_2(q)$ (see \cite{Kleidman}).}\label{Tab3}
\begin{tabular}{llll}\medskip
Subgroup & Condition &Index\\\hline
$D_{(q-1)}$ &$q\geq 13,$ odd&$\frac{1}{2}q\Phi_2$\\
$D_{2(q-1)}$ &$q$ even&$\frac{1}{2}q\Phi_2$\\
$D_{(q+1)}$ &$q\neq 7,9,$ odd&$\frac{1}{2}q\Phi_1$\\
$D_{2(q+1)}$ &$q$ even&$\frac{1}{2}q\Phi_1$\\
Borel subgroup &&$\Phi_2$\\
$\PSL_2(q_0)(2,\alpha)$ &$q=q_0^\alpha$&\\
$S_{4}$ &$q=p\equiv \pm 1 \bmod{8}$&\\
        &$q=p^2,3<p\equiv \pm 3 \bmod{10}$&\\
$A_{4}$ &$q=p\equiv \pm 3 \bmod{8}, q>3$&\\
$A_{5}$ &$q=p\equiv \pm 1 \bmod{10}$&\\
 &$q=p^2,p\equiv \pm 3 \bmod{10}$&\\
  \hline
\end{tabular}
\end{table}

\begin{table}[h]\caption{Maximal subgroups of $\PSL_3(q)$ (see \cite{Kleidman}).}\label{Tab1}
\begin{tabular}{llll}\medskip
Subgroup & Condition \\\hline
$^\wedge[q^2]:\GL_2(q)$ &\\
 $^\wedge(\ZZ_{q-1})^2.S_3$&$q\geq5$\\
$^\wedge \ZZ_{q^2+q+1}.3$&$q\neq4$\\
$\PSL_3(q_0).((q-1,3),b)$&$q=q_0^b$, $b$ prime\\
 $3^2.\SL_2(3)$ &$q=p\equiv1\bmod{9}$\\
  $3^2.Q_8$& $q=p\equiv4,7\bmod{9}$\\
   $\mathrm{SO}_3(q)$ &$q$ odd\\

       $\mathrm{PSU}_3(q_0)$&$q=q_0^2$ \\
       $A_6$&$p\equiv1,2,4,7,8,13\bmod{15}$\\
        $\PSL_2(7)$&$2<q=p\equiv 1,2,4\bmod{7}$\\
        \hline
\end{tabular}
\end{table}

\begin{table}[h]\caption{Maximal subgroups of $\mathrm{PSp}_4(q)$ (see \cite{Kleidman}).}\label{Tab2}
\begin{tabular}{llll}\medskip
Subgroup & Condition \\\hline
$^\wedge[q^3]:(\ZZ_{q-1}\circ \textrm{Sp}_2(q))$ &  $$\\
 $^\wedge[q^3]:\textrm{GL}_2(q)$&\\
$(\textrm{Sp}_2(q)\circ \textrm{Sp}_2(q)).2$&$q\geq3$\\
$^\wedge \textrm{GU}_2(q).2$&$q$ odd, $q\geq5$\\
 $\textrm{PSp}_2(q^2).2$&\\
  $^\wedge \textrm{GL}_2(q).2$& $q$ odd, $q\geq5$\\
   $\textrm{PSp}_4(q_0).(b,(q-1,2))$ &$q=q_0^b$, $b$ prime\\

       $2^4.\Omega_4^-(2)$&$q=p\equiv \pm3 \bmod{8}$\\
$2^4.\Omega_4^-(2)$&$q=p\equiv \pm1 \bmod{8}$\\

       $\textrm{O}_4^+(q)$&$q$ even\\
       $\textrm{O}_4^-(q)$&$q$ even\\
        $\textrm{Sz}(q)$&$q\geq8$, $q$ even, $\log_p(q)$ odd\\
         $\textrm{PSL}_2(q)$&$p\geq5$, $q\geq7$\\

         $S_6$&$q=p\equiv\pm1 \bmod{12}$\\
          $A_6$&$q=p\equiv2,\pm5 \bmod{12}$\\
\hline
\end{tabular}
\end{table}

\begin{lemma}\label{eliminate maximal subgroups} With the above
notation, one of the following happens.
\begin{enumerate}
\item[(i)] $U/M\cong P_a\cong \tau(^\wedge[q^3]:\GL_3(q))$,
$|G':U|=\Phi_2\Phi_4$ and for every $i,$ $t\phi_i(1)$ divides some
member of the set $\mathcal{A}_1,$ where
$$\mathcal{A}_1=\{q\Phi_2,\Phi_1^2\Phi_2,\Phi_1\Phi_3,q\Phi_3,q^3,\Phi_2\Phi_3\}.$$

\item[(ii)] $U/M\cong \tau(^\wedge\mathrm{Sp}_4(q)\cdot (q-1,2))$, $|G':U|=q^2\Phi_1\Phi_3/(2,q-1),$ and for every $i,$ $t\phi_i(1)$
divides $(2,q-1)\Phi_1.$

\item[(iii)] $U/M\cong P_b\cong \tau(^\wedge[q^4]:(\SL_2(q)\times
\SL_2(q)).\ZZ_{q-1})$, $|G':U|=\Phi_3\Phi_4$ and for every $i,$
$t\phi_i(1)$ divides some member of the set $\mathcal{A}_2,$ where
$$\mathcal{A}_2=
\{q\Phi_1,\Phi_1^2,q\Phi_2,\Phi_2^2,\Phi_1\Phi_2,q^2\}.$$
\end{enumerate}
\end{lemma}

\begin{proof} Recall that, for every $i$, $t\phi_i|G:U|$ divides
some  degree of $\PSL_4(q)$. By a routine check on
Tables~\ref{character} and~\ref{maximal-subgroups}, we see that one
of the mentioned cases must happen.
\end{proof}

The next three lemmas can be seen directly from the list of maximal
subgroups of the corresponding groups. We leave the detailed proofs
to the reader.

\begin{lemma}\label{maxSL2} If $K$ is a maximal subgroup of ${\rm{SL}}_2(q)$ whose index  divides
 $\Phi_1,\Phi_2,$ or $q,$ then $K$ is the Borel subgroup of index $\Phi_2$. Moreover, $\Phi_2$ is the smallest index of a proper subgroup of $\rm{SL}_2(q).$
\end{lemma}

\begin{lemma}\label{maxSL3} If $K$ is a maximal subgroup of ${\rm{SL}}_3(q)$ whose index  divides some member of the set
$\mathcal{A}_1$ consisting of
$$q\Phi_2,\Phi_1^2\Phi_2,\Phi_1\Phi_3,q\Phi_3,q^3,\Phi_2\Phi_3,$$ then $K\cong [q^2]:{\rm{GL}}_2(q)$ with index
$\Phi_3.$
\end{lemma}

\begin{lemma}\label{maxSp4} The smallest index of a maximal
subgroup of $\mathrm{Sp}_4(q)$ is $\Phi_4,$ which is greater than
$2\Phi_1$.
\end{lemma}

\begin{lemma}\label{lem1} Let $N\unlhd A$ be such that $A/N\cong
\hat{S},$ where $\hat{S}/Z(\hat{S})\cong \PSL_2(q)$ and let
$\lambda\in\Irr(N).$ If $\chi(1)\leq \Phi_2$ for any
$\chi\in\Irr(A|\lambda),$ then $\lambda$ is $A$-invariant.
\end{lemma}

\begin{proof} By way of contradiction, assume that $\lambda$ is not
$A$-invariant. Let $J=I_A(\lambda)$ and $K\leq A$ such that $J\leq
K\lneq A,$ where $K/N$ is a maximal subgroup of $A/N\cong \hat{S}.$
Let $\delta\in\Irr(J|\lambda)$ and $\mu=\delta^A\in\Irr(A|\lambda).$
Then $$\mu(1)=|A:K||K:J|\delta(1).$$ As $|A:K|\geq \Phi_2$ by
Lemma~\ref{maxSL2} and that $\mu(1)\leq \Phi_2,$ we deduce that
$|K:J|\delta(1)=1$. This implies that $K/N=J/N$ is isomorphic to the
Borel subgroup of $\hat{S}$ of index $\Phi_2$ and that all the
irreducible constituents of $\lambda^J$ are linear. By applying
Gallagher's~Lemma, we deduce that the Borel subgroup $J/N$ is
abelian, which is impossible as $q\geq 13.$
\end{proof}

\begin{lemma}\label{Pb} Let $N\unlhd A$ be such that $A/N\cong \SL_2(q)\circ
\SL_2(q)$ and let $\lambda\in\Irr(N).$ If $\chi(1)\leq \Phi_2^2$ for
any $\chi\in\Irr(A|\lambda),$ then $\lambda$ is $A$-invariant.
\end{lemma}

\begin{proof} Let $A_1$ and $A_2$ be normal subgroups of $A$ such
that $A_i/N\cong \SL_2(q)$ for $i=1,2.$ By way of contradiction,
assume that $\lambda$ is not $A$-invariant. As $A=A_1A_2,$ $\lambda$
is not $A_i$-invariant for some $i.$ Without loss, we assume that
$\lambda$ is not $A_1$-invariant. Let $V=A_1.$ Then $\lambda$ is not
$V$-invariant and $A/V\cong \PSL_2(q).$ Let $I=I_V(\lambda)$ and
$K\leq A$ such that $I\leq K\lneq V,$ where $K/N$ is a maximal
subgroup of $V/N\cong \SL_2(q).$ Let $\delta\in\Irr(I|\lambda)$ be
such that $\delta(1)$ is maximal and let
$\mu=\delta^V\in\Irr(V|\lambda).$ Then
$$\mu(1)=|V:K||K:I|\delta(1).$$ With the assumption on $q,$ we obtain
that $|V:K|\geq \Phi_2$ by Lemma~\ref{maxSL2} and thus $$\mu(1)\geq
|K:I|\Phi_2\delta(1).$$

If $\mu$ is not $A$-invariant, then with the same argument as above,
we deduce that $$\varphi(1)\geq \Phi_2\mu(1)\geq
|K:I|\Phi_2^2\delta(1)$$ for any irreducible constituent $\varphi$
of $\mu^A.$ As $\varphi\in\Irr(A|\lambda),$ we have that
$\varphi(1)\leq \Phi_2^2$ and thus $|K:I|\delta(1)=1.$ It follows
that $K/N=I/N$ is isomorphic to the Borel subgroup of $\SL_2(q)$ of
index $\Phi_2$ and every irreducible constituent of $\lambda^I$ is
linear. By applying Gallagher's Lemma, we deduce that $I/N$ is
abelian which is impossible as $q\geq 13.$

Therefore we conclude that $\mu$ is $A$-invariant. Write
$$\mu^A=\sum_{i=1}^k f_i\chi_i, \text{ where } \chi_i\in\Irr(A|\mu).$$ If
$f_j=1$ for some $j,$ then $\mu$ extends to $\mu_0\in\Irr(A)$ and so
by applying Gallagher's Lemma, $\mu^A$ has an irreducible
constituent of degree $\Phi_2\mu(1)$ since $\Phi_2\in\cd(A/V)$.
Using the same argument as in the previous case, we obtain a
contradiction. Hence we have that all $f_i>1$ and so $f_i's$ are
degrees of nontrivial proper projective irreducible representations
of $A/V\cong \PSL_2(q)$ by Lemma~\ref{lemmaHuppert2}.

If $f_j=q+1$
for some $j,$ then $\chi_j(1)=f_j\mu(1)=\Phi_2\mu(1),$ which leads
to a contradiction as above.

Assume next that $f_j=q-1$ for some $j.$ If $|K:I|\delta(1)\geq 2,$
then as $q\geq 13,$ we obtain
\[\chi_j(1)=(q-1)\mu(1)\geq (q^2-1)|K:I|\delta(1)\geq
2(q^2-1)>(q+1)^2=\Phi_2^2,\] which violates the hypothesis. Thus
$|K:I|\delta(1)=1.$ It follows that $K=I$ and $\delta(1)=1.$ By the
choice of $\delta,$ we deduce that all irreducible constituents of
$\lambda^I$ are linear and so $I/N=K/N$ is an abelian maximal
subgroup of $V/N\cong \SL_2(q).$ It follows that $\PSL_2(q)$ has an
abelian maximal subgroup, which is impossible since $q\geq 13.$

Thus for all $i,$ we have that $f_i>1$ and $f_i\neq q\pm 1.$ By
inspecting the character degrees of $\SL_2(q)$, we conclude that $q$
is odd and $f_i=(q-\epsilon)/2$ for all $i,$ where $q\equiv
\epsilon$ (mod $4$) and $\epsilon=\pm 1.$
Note that when $q\geq 13$ is odd, the characters of degree $q$ or $(q+\epsilon)/2$ of
$\SL_2(q)$ are unfaithful so that these degrees are not the degrees
of the proper projective irreducible representations of $\PSL_2(q).$
Now
\[\frac{1}{2}q(q^2-1)=|\PSL_2(q)|=|A/V|=\sum_{i=1}^k f_i^2=k(\frac{q-\epsilon}{2})^2.\] We obtain that
\[k(q-\epsilon)=2q(q+\epsilon).\] Hence $q-\epsilon\mid
2(q+\epsilon)=2(q-\epsilon)+4\epsilon$ and thus $q-\epsilon$ divides
$4,$ which is impossible as $q-\epsilon\geq q-1\geq 12.$
\end{proof}

\begin{lemma}\label{Pa} Let $N\unlhd A$ be such that $A/N\cong \SL_3(q)$ and let $\lambda\in\Irr(N).$ If
$\chi(1)$ divides some number in
$$\mathcal{A}_1=\{q\Phi_2,\Phi_1^2\Phi_2,\Phi_1\Phi_3,q\Phi_3,q^3,\Phi_2\Phi_3\}$$
for any $\chi\in\Irr(A|\lambda),$ then $\lambda$ is $A$-invariant.
\end{lemma}

\begin{proof}
By way of contradiction, assume that $\lambda$ is not $A$-invariant
and let $J=I_A(\lambda).$ Write
$$\lambda^J=\delta_1+\delta_2+\cdots, \text{ where }
\delta_i\in\Irr(J|\lambda).$$ We have that for every $i,$
$\delta_i^A\in\Irr(A|\lambda)$ and $\delta_i^A(1)=|A:J|\delta_i(1)$
divides some number in $\mathcal{A}_1.$ Let $K$ be a subgroup of $A$
such that $J\leq K$ and that $K/N$ is maximal in $A/N.$ It follows
that the index $|A:K|$ divides some number in $\mathcal{A}_1.$ By
Lemma~\ref{maxSL3}, we deduce that $K/N\cong [q^2]:\textrm{GL}_2(q)$
and $|A:K|=\Phi_3$ so that $$|K:J|\delta_i(1) \text{ divides }
\Phi_1,\Phi_2 \text{ or } q.$$ Let $L$ and $V$ be subgroups of $K$
such that $V/N\cong \SL_2(q)$ and $L/N\cong [q^2].$ Also let
$W=LV\unlhd K.$ Observe that if $\chi\in\Irr(W|\lambda),$ then
$\chi(1)$ must divide $\Phi_1,\Phi_2$ or $q.$ In particular, if
$\chi\in\Irr(V|\lambda),$ then $\chi(1)\leq \Phi_2$ by using
Frobenius reciprocity. By Lemma \ref{lem1}, we obtain that $\lambda$
is $V$-invariant. As the Schur multiplier of $V/N\cong \SL_2(q)$
with $q\geq 13$ is trivial, we deduce from
\cite[Theorem~11.7]{Isaacs} that $\lambda$ extends to
$\lambda_0\in\Irr(V).$ By Gallagher's Lemma,
$\tau\lambda_0\in\Irr(V|\lambda)$ for any $\tau\in\Irr(V/N).$ Choose
$\tau\in\Irr(V/N)$ with $\tau(1)=\Phi_2$ and let
$\mu=\tau\lambda_0.$ Then
$\mu(1)=\tau(1)\lambda(1)=\Phi_2\lambda(1).$

Observe that if $\varphi\in\Irr(W|\mu),$ then $\mu(1)\leq
\varphi(1)\leq \Phi_2$ so that $\lambda(1)=1$ and $\varphi_{V}=\mu$
and thus by \cite[Lemma~12.17]{Isaacs}, we have that
$\textbf{V}(\mu)\unlhd W,$ where $\textbf{V}(\mu)\unlhd V$ is the
vanishing-off subgroup of $\mu.$ As $V$ is non-normal in $W,$ we
deduce that $\textbf{V}(\mu)\lneq V$ and thus as
$\mu_N=\Phi_2\lambda,$ $\mu$ is nonzero on $N,$ and so $N\leq
\textbf{V}(\mu)\lneq V.$ Hence $\textbf{V}(\mu)/N$ lies inside the
center of $V/N.$ In particular,
\begin{equation}\label{equation4}|V:\textbf{V}(\mu)| \text{ is divisible by }
q(q^2-1)/(2,q-1).\end{equation}

On the other hand, by~\cite[Lemma~2.29]{Isaacs}, as $\mu$ vanishes
on $V-\textbf{V}(\mu)$ we have
$[\mu_{\textbf{V}(\mu)},\mu_{\textbf{V}(\mu)}]=|V:\textbf{V}(\mu)|$.
Suppose that $\mu_{\textbf{V}(\mu)}=e\sum_1^t\theta_i$ where the
$\theta_i\in\Irr(\textbf{V}(\mu))$ are distinct and of equal degree
by Clifford theory. We then have
$$|V:\textbf{V}(\mu)|=[\mu_{\textbf{V}(\mu)},\mu_{\textbf{V}(\mu)}]=e^2t \text{ and }
\mu(1)=et\theta_i(1)$$ for all $1\leq i\leq t$. Therefore,
$|V:\textbf{V}(\mu)|\leq \mu(1)^2=\Phi_2^2$, which
violates~\eqref{equation4}. The proof is complete.
\end{proof}


We are now ready to eliminate the three remaining possibilities
singled out in Lemma~\ref{eliminate maximal subgroups}.

\medskip

(i) {\bf Case} $\mathbf {U/M\cong P_a\cong
\tau(^\wedge[q^3]:\GL_3(q)).}$ Then $|G':U|=\Phi_2\Phi_4$ and for
every $i,$ $t\phi_i(1)$ divides some member of the set
$\mathcal{A}_1,$ where
$$\mathcal{A}_1=\{q\Phi_2,\Phi_1^2\Phi_2,\Phi_1\Phi_3,q\Phi_3,q^3,\Phi_2\Phi_3\}.$$
Let $L,V$ and $W$ be subgroups of $U$ such that
$$L/M\cong [q^3],V/M\cong \SL_3(q), \text{ and } W=LV\unlhd U.$$ We
obtain that if $\chi\in\Irr(W|\theta),$ then $\chi(1)$ divides some
number in $\mathcal{A}_1.$

(a) {Subcase} $L\leq I\leq U.$ We have $W/L\cong \SL_3(q).$ Write
$\theta^L=\sum_{i=1}^k\lambda_i,$ where
$\lambda_i\in\Irr(L|\theta).$

We first show that $L/\Ker\theta$ is abelian. It suffices to show
that $L'\leq \Ker\lambda_i$ for all $i,$ and hence \[L'\leq
\cap_{i=1}^k \Ker\lambda_i=\Ker\theta^L=\Ker\theta.\] By way of
contradiction, assume that $L/\Ker\theta$ is nonabelian. Then
$L'\nleq \Ker\lambda_j$ for some $j,$ hence
$\lambda_j(1)>\theta(1)=1$ and $p\mid \lambda_j(1).$ As
$\Irr(W|\lambda_j)\subseteq \Irr(W|\theta),$ by Lemma~\ref{Pa}, we
obtain that $\lambda_j$ is $W$-invariant.

As $q\geq 8,$ we deduce that the Schur multiplier of $W/L\cong
SL_3(q)$ is trivial so that by \cite[Theorem~11.7]{Isaacs},
$\lambda_j$ extends to $\lambda_0\in\Irr(W)$ and hence by Gallgher's
Lemma, $\lambda_0\tau$ are all the irreducible constituents of
$\lambda^W$ where $\tau\in\Irr(W/L).$ By choosing $\tau\in\Irr(W/L)$
with $\tau(1)=q^3,$ we obtain that
$\lambda_0(1)\tau(1)=q^3\lambda_j(1)$ must divide some number in
$\mathcal{A}_1,$ which is impossible as $\lambda_j(1)>1.$

Thus $L/\Ker\theta$ is abelian. By Lemma~\ref{Pa} again, all
$\lambda_i$ are $W$-invariant and linear. Therefore $$[W,L]\leq
\cap_{i=1}^k \Ker\lambda_i\leq \Ker\theta^L=\Ker\theta\leq M,$$
which is a contradiction as $W$ acts nontrivially on $L/M\cong
[q^3].$

(b) Subcase $L\nleq I.$ As $I\cap L\lneq L,$ we can choose
$\lambda\in\Irr(L|\theta)$ with $p\mid \lambda(1).$ By Lemma
\ref{Pa}, we obtain that $\lambda$ is $W$-invariant. As $q\geq 13,$
we deduce that the Schur multiplier of $W/L\cong \SL_3(q)$ is
trivial so that by \cite[Theorem~11.7]{Isaacs}, $\lambda$ extends to
$\lambda_0\in\Irr(W)$ and hence by Gallgher's Lemma, $\lambda_0\tau$
are all the irreducible constituents of $\lambda^W$ where
$\tau\in\Irr(W/L).$ By choosing $\tau\in\Irr(W/L)$ with
$\tau(1)=q^3,$ we obtain that $\lambda_0(1)\tau(1)=q^3\lambda(1)$
must divide some number in $\mathcal{A}_1,$ which is impossible as
$\lambda(1)>1.$

\medskip

(ii) {\bf Case} $\mathbf {U/M\cong \tau(^\wedge\textrm{Sp}_4(q)\cdot
(q-1,2))}\cong \textrm{S}_4(q)\cdot [(q-1,2)^2/(q-1,4)]$. Then
$|G':U|=q^2\Phi_1\Phi_3/(2,q-1),$ and for every $i,$ $t\phi_i(1)$
divides $(2,q-1)\Phi_1.$  Let $M\unlhd W\unlhd U$ be such that
$W/M\cong \textrm{S}_4(q).$ As $W\unlhd U,$ for any
$\varphi\in\Irr(W|\theta),$ we obtain that $\varphi(1)$ divides
$(2,q-1)\Phi_1.$ By Lemma~\ref{maxSp4}, we deduce that $\theta$ is
$W$-invariant. Write $\theta^W=\sum_{i=1}^k f_i\mu_i,$ where $\mu_i
\in\Irr(W|\theta).$ If $f_j=1$ for some $j,$ then $\theta\in\Irr(M)$
is extendible to $\theta_0\in\Irr(W).$ By Lemma~\ref{Gallagher} we
obtain that $\tau\theta_0$ are all the irreducible constituents of
$\theta^W$ for $\tau\in\Irr(W/M).$ As $W/M\cong \textrm{S}_4(q),$ it
has an irreducible character $\tau\in\Irr(W/M)$ with $\tau(1)=q^4,$
and hence $\theta_0(1)\tau(1)=q^4$ must divide $(2,q-1)\Phi_1,$
which is impossible. Hence this case cannot happen. Thus all $f_i>1$
and they are the degree of projective irreducible representation of
$W/M\cong \textrm{S}_4(q).$ As $W\unlhd U,$ we deduce that
$\varphi(1)$ divides $(2,q-1)\Phi_1$ for any
$\varphi\in\Irr(W|\theta).$ Therefore as
$\mu_i(1)=f_i\theta(1)=f_i,$ we deduce that $f_i(1)$ divide
$(2,q-1)\Phi_1$ for all $i.$ In particular, we have $f_i\leq
2\Phi_1$ for all $i.$ However using \cite{Lubeck}, we see that the
smallest nontrivial projective degree of $S_4(q)$ when $q>3$ is
$\frac{1}{2}\Phi_1\Phi_2$ which is obviously larger then $2\Phi_1$
as $q\geq 13.$ Thus this case cannot happen.
\medskip

(iii) {\bf Case} $\mathbf {U/M\cong
P_b\cong\tau(^\wedge[q^4]:(\textrm{SL}_2(q)\times
\textrm{SL}_2(q)).\ZZ_{q-1}))}$. Then $|G':U|=\Phi_3\Phi_4$ and for
every $i,$ $t\phi_i(1)$ divides some member of the set
$$\mathcal{A}_2=
\{q\Phi_1,\Phi_1^2,q\Phi_2,\Phi_2^2,\Phi_1\Phi_2,q^2\}.$$ Let
$L,V_1,V_2$ be subgroups of $U$ containing $M$ such that $$L/M\cong
[q^4], V_i/M\cong \textrm{SL}_2(q) \text{ for } i=1,2,$$ and let
$$V=V_1V_2,W=LV.$$ Then $W\unlhd U$ and for any
$\varphi\in\Irr(W|\theta),$ we obtain that $\varphi(1)$ divides some
number in $\mathcal{A}_2.$ In particular, $\varphi(1)\leq \Phi_2^2.$
By Frobenius reciprocity and the fact that $(\theta^V)^W=\theta^W,$
we obtain that $\chi(1)\leq \Phi_2^2$ for any
$\chi\in\Irr(V|\theta).$ By Lemma \ref{Pb}, we have that $\theta$ is
$V$-invariant and hence $V\leq I.$ Let $$L_1=L\cap I \text{ and }
X=W\cap I.$$ We have that $L_1\unlhd X$ and $W=LX$ since $W=LV$ and
$V\leq I.$

(a) {Subcase} $L\leq I.$ Then $\theta$ is $W$-invariant.  We have
$W/L\cong \textrm{SL}_2(q)\circ \textrm{SL}_2(q)$ and $W$ acts
irreducibly on $L/M\cong [q^4].$ Write
$$\theta^L=\lambda_1+\lambda_2+\cdots+\lambda_k, \text{ where } \lambda_i\in\Irr(L|\theta).$$ We will show that $L/ker\theta$ is
abelian. It suffices to show that $L'\leq \Ker\lambda_i$ for all
$i,$ and hence \[L'\leq \cap_{i=1}^k
\Ker\lambda_i=\Ker\theta^L=\Ker\theta.\] By way of contradiction,
assume that $L/\Ker\theta$ is nonabelian. Then $L'\nleq
\Ker\lambda_j$ for some $j,$ hence $\lambda_j(1)>\theta(1)=1$ and
$p\mid \lambda_j(1).$ By Lemma~\ref{Pb}, we deduce that $\lambda_j$
is $W$-invariant. By~\cite[Theorem~$2$]{ABS}, we deduce that $L/M$
is an abelian chief factor of $W,$ and so by Lemma~\ref{fully
ramified}, we obtain that $\lambda_j(1)=q^{2}.$ Hence for any
$\delta\in\Irr(W|\lambda_j),$ we have that $q^{2}\mid \delta(1)$ and
$\delta(1)$ divides one of the number in $\mathcal{A}_2,$ so that
$\delta(1)= q^2.$ By~\cite[Theorem~2.3]{Moreto}, we obtain a
contradiction as $W/L\cong \SL_2(q)\circ \SL_2(q)$ is nonsolvable.

Thus $L/\Ker\theta$ is abelian. Hence $\lambda_i(1)=1$ and also
$\lambda_i$ are $W$-invariant for all $i$ by Lemma~\ref{Pb}. Hence
we obtain that
$$[W,L]\leq \cap_{i=1}^k \Ker\lambda_i\leq \Ker\theta^L=\Ker\theta\leq
M,$$ which is a contradiction as $W$ acts nontrivially on $L/M\cong
[q^4].$

(b) {Subcase} $L\nleq I.$ Then $L_1/M\lneq L/M.$ We have that
$L_1/M\unlhd X/M$ and since $L/M\cong [q^4]$ is abelian, we also
have $L_1/M\unlhd L/M,$ and hence \[L_1/M\unlhd XL/M=W/M.\] As
$L_1/M\lneq L/M$ and $L/M$ is irreducible under $W/M,$ we deduce
that $L_1=M.$ Then \[t=|U:WI|\cdot |W:X|=|U:WI|\cdot
|L:L_1|=|U:WI|\cdot |L:M|=q^4|U:WI|.\] Hence $q^4\mid t,$ which is
impossible as $t$ divides some number in $\mathcal{A}_2.$


\section{Final step}\label{step 5}
Recall from section~\ref{step3} that
$$G/M\cong G'/M\times C/M,$$ where $C/M=C_{G/M}(G'/M)$ and
$G'/M\cong\PSL_4(q)$.

We first claim that $C/M$ is abelian. Assume not, then it would have
a nonlinear irreducible character, say $\psi$. It follows that
$\psi\times\St_{G'/M}$ is an irreducible character of $G/M$ of
degree $q^6\psi(1)$, which is impossible.

\medskip

(i) We assume that $M=M'.$ If $M$ is abelian then $M=M'=1$. It
follows that $G\cong G'\times C_G(G')$ and we are done as $C_G(G')\cong G/G'$
is abelian. So it remains to consider the case when $M$ is
non-abelian. Let $N\leq M$ be a normal subgroup of $G'$ so that
$M/N\cong S^k$ for some non-abelian simple group $S$.
By~\cite[Lemma~4.2]{Moreto1}, $S$ has a non-principal irreducible
character $\varphi$ extending to $\Aut(S)$. Lemma~\ref{lemma2} and
its proof then imply that $\varphi^k$ extends to $G'/N$. Therefore,
by Gallagher's lemma, $\varphi^k\St_{G'/M}\in\Irr(G'/N)$ where
$\St_{G'/M}$ is the Steinberg character of $G'/M\cong\PSL_4(q)$.
However, $\varphi^k(1)\St_{G'/M}(1)$, which is $q^6\varphi^k(1)$,
does not divide any degree of $G$ and hence we get a contradiction.

\medskip

(ii) Now we can assume $M>M'$. As shown in section~\ref{step4},
every linear character of $M$ is $G'$-invariant. Using
Lemma~\ref{lemmaHuppert1}, we deduce that
$$[M,G']=M' \text{ and } |M:M'|\mid |\Mult(G'/M)|.$$ Remark that $\Mult(\PSL_4(q))$ is a cyclic
group of order $(q-1,4)$. We see that $q$ must be odd and we come up
with the two following cases:

\textbf{Case} $\mathbf{q\equiv 3\bmod{4}}$: Then $|\Mult(G'/M)|=2$
and hence $|M:M'|=2$. Therefore, by Lemma~\ref{Schur cover}, we have that
$G'/M'\cong\SL_4(q).$ For any $x,y\in G'$ and $c\in C$, the fact
$[M,G']=M'$ implies that
$$[xM',yM']^{cM'}=[xM'^{cM'},yM'^{cM'}]=
[xM'mM',yM'nM']=[xM',yM']$$ for some $m,n\in M$. In other words, $C$
centralizes $G'/M'$ and we obtain that $G/M'$ is an internal central
product of $G'/M'$ and $C/M'$ with amalgamated $M/M'\cong \ZZ_2$.

We claim that if $C/M'$ is abelian then $\cd(G/M')=\cd(G'/M')$. The
inclusion $\cd(G/M')\subseteq \cd(G'/M')$ is obvious. Thus it
remains to show that $\cd(G'/M')\subseteq \cd(G/M')$. Let $tM'$ be
the generator of $M/M'$ and $\chi\in\Irr(G'/M')$. As $tM'$ is a
central involution of $G'/M'$, we have $\chi(tM')=\chi(1)$ or
$-\chi(1)$. Since $tM'$ is also a central involution of the abelian
group $C/M'$, there exists a linear character $\theta\in\Irr(C/M')$
such that $\theta(tM')=\chi(1)/\chi(tM')$. We then have
$$\chi\times\theta\in\Irr(G'/M'\times C/M')$$ and furthermore $$(tM',tM')\in\Ker(\chi\times\theta).$$
It follows that $\chi\times\theta$ can be considered as a character
of $G/M'$ and hence we have proved that $\chi(1)\in\cd(G/M')$, as
wanted.

From the above claim, if $C/M'$ is abelian, we would have
$$\cd(SL_4(q))=\cd(G'/M')=\cd(G/M')\subseteq\cd(G)=\cd(\PSL_4(q)).$$
This however contradicts Lemma~\ref{cd(SL4) is not in cd(PSL)} that
$\Phi_1\Phi_2\Phi_3\Phi_4/2$ is in $\cd(\SL_4(q))$ but not in
$\cd(\PSL_4(q))$. Thus we obtain that $C/M'$ is non-abelian. Let
$\alpha$ be a nonlinear irreducible character of $C/M'$. As $C/M$ is
abelian, $\alpha$ must be faithful and hence
$\alpha(tM')=-\alpha(1)$. On the other hand, by inspecting the
character degrees and their multiplicities of $\SL_4(q)$ and
$\PGL_4(q)$ from~\cite{Lubeck}, one sees that both have degree
$\Phi_1^2\Phi_2^2\Phi_4$ with multiplicity $q(q^2-1)/3$. This in
particular implies that $G'/M'\cong\SL_4(q)$ has a faithful
irreducible character of degree $\Phi_1^2\Phi_2^2\Phi_4$, say
$\beta$. We then have
$$\beta(tM')=-\beta(1)=-\Phi_1^2\Phi_2^2\Phi_4.$$
Therefore,
$$\alpha(tM')\beta(tM')=\alpha(1)\beta(1).$$ In other words,
$(tM',tM')\in\Ker(\beta\times\alpha)$. As $G/M'$ is an internal
central product of $G'/M'$ and $C/M'$ with amalgamated $M/M'$, we
obtain
$$\beta\times\alpha\in\Irr(G/M').$$ This however is a
contradiction since
$\alpha(1)\beta(1)=\Phi_1^2\Phi_2^2\Phi_4\alpha(1)>\Phi_1^2\Phi_2^2\Phi_4$
and $G$ has no degree which is a proper multiple of
$\Phi_1^2\Phi_2^2\Phi_4$.

\textbf{Case} $\mathbf{q\equiv 1\bmod{4}}$: Then $|\Mult(G'/M)|=4$
and hence $|M:M'|=2$ or $4$. The case $|M:M'|=2$ can be handled
similarly as above with the notice that $$G'/M'\cong
\SL_4(q)/\langle -I\rangle,$$ where $I$ is the identity $4\times4$
matrix. So we assume from now on that $|M:M'|=4$. It follows by
Lemma~\ref{Schur cover} that
$G'/M'\cong\SL_4(q).$ Arguing as in the case $q\equiv 3\bmod{4}$, we obtain that $G/M'$ is an internal central
product of $G'/M'$ and $C/M'$ with amalgamated $M/M'\cong \ZZ_4$. If
$C/M'$ is abelian, as above, we would have
$\cd(SL_4(q))\subseteq\cd(\PSL_4(q)).$ This again is impossible by
Lemma~\ref{cd(SL4) is not in cd(PSL)}. Therefore we have that $C/M'$
is non-abelian. Let $\alpha$ be a nonlinear irreducible character of
$C/M'$. Then $\alpha(tM')=\kappa\alpha(1)$ where $\kappa$ is a
nontrivial fourth root of $1$ and $tM'$ is the generator of $M/M'$.
Inspecting the character degrees and their multiplicities of
$\SL_4(q)$ and $\PGL_4(q)$ from~\cite{Lubeck}, one sees that both
have degree $\Phi_1^2\Phi_2^2\Phi_4$ with multiplicity $q(q^2-1)/3$.
Also, $\SL_4(q)$ has no degree $\Phi_1^2\Phi_2^2\Phi_4/2$ or
$\Phi_1^2\Phi_2^2\Phi_4/4$. This in particular implies that
$G'/M'\cong\SL_4(q)$ has a faithful irreducible character of degree
$\Phi_1^2\Phi_2^2\Phi_4$, say $\beta$, such that
$$\beta(tM')=\kappa^3\beta(1)=\kappa^3\Phi_1^2\Phi_2^2\Phi_4.$$
Therefore,
$$\alpha(tM')\beta(tM')=\kappa^4\alpha(1)\beta(1)=\alpha(1)\beta(1).$$ In other words,
$tM'\in\Ker(\alpha\beta)$ and hence
$$\alpha\beta\in\Irr(G'/M'\circ C/M')=\Irr(G/M').$$ This is a
contradiction as
$(\alpha\beta)(1)=\Phi_1^2\Phi_2^2\Phi_4\alpha(1)>\Phi_1^2\Phi_2^2\Phi_4$
and $G$ has no degree which is a proper multiple of
$\Phi_1^2\Phi_2^2\Phi_4$.


\section*{Acknowledgement} The first author gratefully acknowledges
the support of the Faculty Research Grant FRG 1747 from The
University of Akron. The second author is supported by NRF (South
Africa) and North-West University (Mafikeng).

\end{document}